\documentclass[final,leqno,onefignum,onetabnum]{siamltex}

\usepackage{mathtools}
\mathtoolsset{showonlyrefs} 
\usepackage[normalem]{ulem} 

\usepackage{amsmath} 
\usepackage{amssymb}
\usepackage{mathrsfs}
\usepackage{nicefrac}

\usepackage[color]{showkeys}
\definecolor{refkey}{gray}{0.8}
\definecolor{labelkey}{gray}{0.8}

\usepackage{subcaption}
\usepackage[colorlinks=true,linkcolor=blue,citecolor=blue,urlcolor=blue,pdfborder={0 0 0},bookmarks,bookmarksnumbered,bookmarksdepth=3]{hyperref}
\usepackage{url}

\usepackage[USenglish]{babel}

\newcommand{\cut}[1]{{}}

\newcommand{\cC}{{\mathcal{C}}}

\newcommand{\cG}{{\mathcal{G}}}
\newcommand{\cH}{{\mathcal{H}}}

\newcommand{\RR}{\mathbb{R}}

\newcommand{\Proj}{{\mathrm{Proj}}}

\newcommand{\prox}{\mathbf{prox}}

\DeclareMathOperator*{\argmin}{arg\,min}

\DeclareMathOperator*{\Min}{minimize}

\DeclareMathOperator*{\zer}{zer}

\DeclarePairedDelimiter{\dotp}{\langle}{\rangle}

\newcommand{\bc}{\begin{center}}
\newcommand{\ec}{\end{center}}

\newcommand{\bdm}{\begin{displaymath}}
\newcommand{\edm}{\end{displaymath}}

\newcommand{\beq}{\begin{equation}}
\newcommand{\eeq}{\end{equation}}

\newcommand{\bfl}{\begin{flushleft}}
\newcommand{\efl}{\end{flushleft}}

\newcommand{\bt}{\begin{tabbing}}
\newcommand{\et}{\end{tabbing}}

\newcommand{\beqn}{\begin{align}}
\newcommand{\eeqn}{\end{align}}

\newcommand{\beqs}{\begin{align*}} 
\newcommand{\eeqs}{\end{align*}}  

\newtheorem{assumption}{Assumption}
\usepackage[pdftex]{graphicx}

\usepackage{multirow}
\usepackage{subcaption}
\usepackage{booktabs}

\usepackage{alphalph}

\usepackage[ruled,vlined,linesnumbered]{algorithm2e}

\newtheorem{remark}{Remark}
\ifpdf
\hypersetup{
  pdftitle={Cyclic Coordinate Update Algorithms for Fixed-Point Problems: Analysis and Applications},
  pdfauthor={Yat Tin Chow, Tianyu Wu, Wotao Yin}
}
\fi

\newcommand{\xin}{y}

\begin{document}

\title{Cyclic Coordinate Update Algorithms for Fixed-Point Problems: Analysis and Applications}
\author{
  Yat Tin Chow\footnotemark[1]
  \and Tianyu Wu\footnotemark[1]
  \and Wotao Yin\footnotemark[1]
}
\footnotetext[1]{
  Department of Mathematics, University of California, Los Angeles, CA 90095, USA, \url{ytchow / wuty11 / wotaoyin@math.ucla.edu}. This research is supported by the NSF grant ECCS-1462397.}
\date{\today}
\maketitle

\begin{abstract}
Many problems reduce to the fixed-point problem of solving $x=T(x)$, {where $T$ is a mapping from a Hilbert space to itself.} To this problem, we apply the coordinate-update algorithms, which update only one or a few components of $x$ at each step. When {each step} is cheap, these algorithms are faster than the full fixed-point iteration (which updates all the components).

In this paper, we focus on {cyclic coordinate} selection rules, where the ordering of coordinates in each cycle is arbitrary. {The corresponding} algorithms are fast, but their convergence is unknown in the fixed-point setting.

When $T$ is a nonexpansive operator and has a fixed point, we show that the sequence of coordinate-update iterates converges to a fixed point under proper step sizes. This result applies to the primal-dual coordinate-update algorithms, which have {wide} applications to optimization problems with nonseparable nonsmooth objectives, as well as global linear constraints.

Numerically, we apply coordinate-update algorithms with cyclic, shuffled cyclic, and random selection rules to $\ell_1$ robust least squares, {total variation minimization}, as well as nonnegative matrix factorization. They converge much faster than the standard fixed-point iteration. Among the three rules, cyclic and shuffled cyclic rules are overall faster than the random rule.
\end{abstract}




\section{Introduction}\label{sec:intro}
We recently witnessed a strong demand for fast numerical solutions for large-scale problems. Numerical methods of small memory footprints become very popular. Among them \emph{coordinate update algorithms} (e.g.,  \cite{bertsekas1983distributed,bertsekas1989parallel,luo1992convergence,tseng2001convergence,nesterov2012rcd,razaviyayn2013unified,xu2013block,combettes2015stochastic,Peng_2015_AROCK,
peng2016coordinate,chen2015extended}) are found to be very useful. They reformulate a problem as a fixed-point problem and decompose it further into simple subproblems, each updating one, or a small block of, variables while fixing others.
They are popular numerical choices {for} problems with the \emph{coordinate-friendly structure}, namely, there are means to update a component, or a small block of components, of its variable much cheaper  than updating all components of the variable. A variety of problems, including second order cone programming, variational image processing, support vector machines, empirical risk minimization, portfolio optimization, distributed computing, and nonnegative matrix factorization, have this structure \cite{peng2016coordinate}. 

A class of coordinate update algorithms is the \emph{coordinate descent algorithms} for optimization, where the objective function, or {a surrogate of the objective function,} is reduced at each iteration; e.g., see \cite{beck2013convergence,razaviyayn2013unified}. In randomized coordinate descent algorithms such as \cite{nesterov2012rcd,Lu_Xiao_rbcd_2015,richtarik2014iteration}, it is the (conditional) expectation of the objective function that descends iteratively.
Coordinate descent algorithms are efficient at solving problems with Lipschitz differentiable objective functions, as well as those with \emph{separable} non-differentiable functions and constraints.  See the recent survey papers \cite{wright2015coordinate,shi2016primer,yuan2010comparison}.

On the other hand, there exist non-separable, non-differentiable examples \cite{warga1963minimizing} \cite[Section 2.2.1]{shi2016primer} which coordinate descent algorithms {fail to solve because they} get stuck at a non-stationary point. Also, it is difficult for coordinate descent algorithms to directly handle problems with global constraints.
Such examples are found in problems such as $\ell_1$-robust least squares, total variation image processing, and the extended monotropic program. They can be solved by the recent primal-dual coordinate-update algorithms \cite{pesquet2014class,combettes2015stochastic,fercoq2015coordinate}~\cite[Section 4]{peng2016coordinate}. Primal-dual coordinate-update algorithms {do not always reduce the objective values monotonically,} so it is challenging to analyze their convergence in the coordinate descent framework. However, {because they are fixed-point algorithms, they can be analyzed by the results of this paper.} 
In fact, {applying cyclic coordinate updates to primal-dual algorithms is} new to the best of our knowledge.

In addition to optimization problems, fixed-point iterations arise in variational inequalities~\cite{xia2011projection}, inverse problems~\cite{hanke1995convergence}, equilibrium analysis~\cite{iusem2010proximal,da2010cone} and control theory~\cite{nong2009parameter}, where the fixed-point operators are nonexpansive. 

This paper focuses on the fixed-point problem with a nonexpansive operator. Despite the rich literature on the fixed-point problem, our understanding to  its coordinate-update algorithms is very limited. In the literature, there are two classes of convergence analysis. They use different metrics {to measure} the distance between the current iterate and the fixed point. The first class uses the (weighted) $\ell_\infty$ distance \cite{bertsekas1983distributed,bertsekas1989parallel}. Their applications include both linear and nonlinear systems with dominant diagonals, as well as certain optimization
problems with a smooth objective and simple constraints. The second class uses the $\ell_2$ distance and has more applications. However, convergence results are limited to the random selection of coordinates~\cite{combettes2015stochastic,Peng_2015_AROCK}\cite[Appendix D]{peng2016coordinate}. 
{On the other hand, the cyclic selection typically shows better practical performance, but its convergence has not been studied under the $\ell_2$ distance yet.}

Let us briefly compare the random and cyclic selection rules. The random selection is easier to analyze since taking expectation reduces it to the analysis of the standard (full update) fixed-point iteration. However, cyclic selection took fewer iterations to converge than random selection in all of our numerical tests (presented in Section~\ref{sec:numerical} below). The advantages of cyclic selection {have} also been observed with coordinate descent algorithms~\cite{friedman2007pathwise,friedman2010regularization}. Cyclic selection is also more cache efficient{, because data} of adjacent coordinates are typically stored at consecutive memory locations.
Cyclic selection accesses these data sequentially, likely encountering \emph{cache hits}. Random selection, on the contrary, accesses data randomly, likely encountering \emph{cache misses}. Hence, random selection is less cache efficient. Moreover, random selection also requires pseudo-random number generation, which can  take more time than the coordinate updates if the underlying problem is coordinate-friendly with very cheap updates.

Our exposition in this paper was motivated by the observations that cyclic coordinate update algorithms have excellent numerical performance for many fixed-point problems yet their convergence has not been understood yet.

\subsection{Problem formulation}
Given an operator $S:\cH\to\cH$, our problem is to find $x\in\cH$ such that
\begin{align}\label{eq:prob}
  Sx=0.
\end{align}
{(We often use $Sx$ to abbreviate $S(x)$.)} This problem is equivalent to finding a fixed point to the operator $T=I-S$, that is, the point satisfies $x=Tx$. (In the rest of this paper we choose to use $S$ instead of $T$ because it is overall more convenient to present our results in $S$.) Our assumption to this problem is that the operator $I-S$ has a fixed point (possibly not unique) and it is nonexpansive, that is,
\begin{align}\label{eq:nonexp}
  \|(I-S)x - (I-S)y\|\le \|x-y\|,\quad\forall x,y\in \cH.
\end{align}
The condition  \eqref{eq:nonexp} is equivalent to $S$ being  $(1/2)$-cocoercive, namely,
\begin{align}\label{eq:cocoer}
  \frac{1}{2}\|Sx - Sy\|^2\le \dotp{Sx-Sy,x-y},\quad\forall x,y\in \cH.
\end{align}
{In the above and throughout the paper, $\|\cdot\|=\sqrt{\langle\cdot,\cdot\rangle}$ is the norm induced by the inner product of the Hilbert space $\cH$.}
\subsection{Contributions}
This paper proposes the standard and shuffled cyclic coordinate update algorithms to solve \eqref{eq:prob} {in} a Hilbert space $\cH$. The contributions of this paper include:
\begin{enumerate}
\item {Assume that} the operator $I-S$ is nonexpansive and has a fixed point. We show that cyclic coordinate-update algorithms converge to a fixed point {under proper step sizes}. This result holds for all orderings of the coordinates as long as {every coordinate is updated in each} cycle. For example, one can shuffle the coordinates or  rank the coordinates based on their coordinate-wise Lipschitz constants.

\smallskip

We show in Theorem~\ref{thm:decreasingalpha2} that a sequence of $O(\tfrac{1}{\sqrt{k}})$ step sizes ensures convergence.  If  the operator $S$ is further \emph{quasi-$\mu$-strongly monotone} (this property is defined Assumption~\ref{assump:qstrong} below and is weaker than the $\mu$-strongly monotone property),
then a fixed step size is sufficient for convergence (Theorem~\ref{thm:constantalpha}). 
{Although fixing the step size to 1 worked throughout our tested problems, the theoretical step size is inversely proportional to the number of coordinates.}

\smallskip

Our proofs are based on comparing the operator of cyclic coordinate update to the standard full update operator and bounding {their differences.} This approach is significantly different from those for cyclic coordinate descent~\cite{beck2013convergence,razaviyayn2013unified} (based on function value descent) and fixed-point random coordinate update~\cite{combettes2015stochastic,Peng_2015_AROCK} (taking expectations and using super-martingale convergence).

\item  Based on the  algorithms in part 1, we propose specific coordinate-update algorithms for two problems: $\ell_1$-robust least squares and total-variation based computer tomograph (CT) reconstruction. We also test an existing algorithm for nonnegative matrix factorization. We briefly explain how to apply coordinate updates by exploiting their coordinate-friendly structures. 

\smallskip

For all three problems, we find that cyclic coordinate updates, either deterministic or shuffled, take fewer epochs\footnote{An epoch consists of $m$ coordinate updates, where $m$ is the total number of coordinates.} to reach the same accuracy than randomized coordinate update, and {every kind of} coordinate {update takes} fewer epochs than (full update) fixed-point iteration.

\smallskip

Our numerical results are obtained with the optimal fixed step size parameters for all algorithms. We conjecture that fixed step sizes sufficiently ensure the convergence of our algorithms when the fixed-point problem is derived from an optimization problem {using} the primal-dual method.
\end{enumerate}

\subsection{Algorithm}
In this paper our variable has $m$ blocks,  $x_i\in \cH_i$, $i=1,\ldots,m$, where $\cH_i$ is a Hilbert space. Their Cartesian product is
$  \cH:=\cH_1\times\cdots\times\cH_m.$

For each coordinate $i=1,\ldots,m$, we define the $i$th coordinate operator $S_i:\cH\to\cH$ such that
\begin{align}\label{eq:Si}
  S_ix=\Big(0,\ldots,(Sx)_i,\ldots,0\Big),\quad\forall x\in\cH.
\end{align}
We establish the convergence of Algorithm \ref{alg:cyccu} and numerically demonstrate its efficiency.
\begin{algorithm}[h!]
  \caption{Cyclic Coordinate Update\label{alg:cyccu}}
  \SetKwInOut{Input}{input}
  \Input{$x^0\in\cH$}
  \For{$k=1,2,\ldots,$}{
    set $(i_1,i_2,...,i_m)$ as a permutation of $(1,2,...,m)$\;
    (either no permutation, random shuffling, or greedy ordering)\;
    {choose a step size $\alpha_k>0$\;}
    initialize $\xin^0\gets x^{k-1}$\;
    \For{$j=1,...,m$}{
      set $\xin^{j} \gets \xin^{j-1} - \alpha_k S_{i_j}(\xin^{j-1})$\;
    }
    set $x^k \gets \xin^m$\;
  }
\end{algorithm}

In Algorithm \ref{alg:cyccu}, each $k$ specifies an outer loop, which is also called an epoch, and each $j$ is an inner loop. Line 5 sets an initial point for the inner loop, then  Line 6--7 update each of the $m$ coordinates once, and finally Line 8 finishes the inner loop by {passing its result to} $x^k$. 
Note that each inner iteration only updates the $i_j$th coordinate:  
\begin{align}
  (\xin^j)_i=
  \begin{cases}
  (\xin^{j-1})_i-\alpha_k (S(\xin^{j-1}))_i, & \mbox{if } i=i_j  \\
  (\xin^{j-1})_i,                            & \mbox{otherwise,}
  \end{cases}\label{eqn:update}
  \quad \mbox{for }i=1,\ldots,m.
\end{align}

The order of the $m$ coordinates is specified at the beginning of each epoch. {Each epoch selects the order independently.}
Typical choices include the {natural} ordering $(1,2,...,m)$, a random shuffling, and a greedy ordering (for example, place a coordinate earlier if the Lipschitz constant of the coordinate is larger).
We can also \emph{shuffle only in the first epoch} and then {use the same} ordering for all remaining epochs. All these ordering rules are numerically efficient, and our analysis applies to all of them.

The non-shuffled cyclic ordering is the easiest to code 
and has the cheapest per-iteration cost because shuffling requires pseudo-random number generation, which can be time consuming, especially when each inner iteration is simple. Shuffling, however, avoids the worst ordering and {may} accelerate the convergence.

\subsection{Organization}
The rest of this paper is organized as follows. Section~\ref{sec:bac} briefly reviews fixed-point iteration of nonexpansive operators and its applications. Section~\ref{sec:analysis} presents our theoretical analysis. Numerical results and further applications are presented in Section~\ref{sec:numerical}. Finally, Section~\ref{sec:con} concludes this paper.

\section{Background}\label{sec:bac}
%
%
A traditional algorithm for Problem~\eqref{eq:prob} is the Krasnosel'ski\u{i}-Mann (KM) iteration~\cite{krasnosel1955two,mann1953mean}:
\begin{align}\label{KM}
  x^{k+1}=x^k-\eta_k Sx^k,
\end{align}
where $\eta_k$ is a step size parameter, {and the operator $S$ is applied to $x^k$.}

KM iteration~\eqref{KM} has many special cases such as gradient descent, proximal-point, prox-gradient, as well as many operator-splitting algorithms including forward-backward splitting, Peaceman-Rachford splitting, Douglas-Rachford splitting \cite{LionsMercier1979_SplittingAlgorithms}, the alternating direction of multipliers (ADMM) \cite{GlowinskiMarroco1975_LapproximationPar,GabayMercier1976_dual}, three-operator splitting \cite{DavisYin2015_threeoperator}, primal-dual splitting \cite{condat2013primal,vu2013splitting}. Many of their variations are special examples of KM iteration, too. For each of these algorithms, one can recover an operator $S$ such that $(I-S)$ is nonexpansive though this recovery is not {always} obvious.
\subsection{Notation and preliminaries}
The solution to our problem \eqref{eq:prob} is the zero set of $S$:
$$\zer(S):=\{x\in\cH: 0=Sx\}.$$
The minimal assumption that we make to the problem is{:}
\begin{assumption}\label{assump:S}
  $\zer(S)$ is nonempty. $(I-S)$ is nonexpansive. 
\end{assumption}

We frequently use the following conventions and properties on operators: for any operators $A,B,C:\cH\to\cH$ and point $x\in \cH$, we have
\begin{align}
ABCx & = A(B(Cx)), \\
(A+B+C)x & = Ax+Bx+Cx,\\
(A+B)Cx & =A Cx + BCx. 
\end{align}
However, unless $A$ is a linear operator, we do not have $A(B+C)x = ABx + ACx$.

As already mentioned, $(I-S)$ is nonexpansive \eqref{eq:nonexp} if and only if $S$ is $(1/2)$-cocoercive \eqref{eq:cocoer}, {which} 
implies that $S$ is 2-Lipschitz:
\begin{align}
  \|Sx-Sy\|\leq2\|x-y\|,\quad\forall x,y\in\cH.\label{lipschitz}
\end{align}
Since $\|S_ix-S_iy\|\le \|Sx-Sy\|$, each $S_i$ is 2-Lipschitz, too. However, the Lipschitz constants for $S_i$ can be (much) smaller than $2$, permitting theoretically larger step sizes.
Hence, \textbf{we let $L_i$ be the Lipschitz constant of $S_i$,} i.e.,
\begin{align}
  \|S_ix-S_iy\|\leq L_i\|x-y\|,\quad\forall x,y\in\cH,
\end{align}
and, for simplicity, let \begin{align} L:=\max_iL_i\le 2.\label{def:L}\end{align}
Properly replacing $L$ by different $L_i$ in our analysis  below {will} improve our results, but we prefer simplicity over better bounds.

Next, we switch our focus to KM iteration and Algorithm \ref{alg:cyccu}. Without loss of generality, we fix the coordinate updating order as $1$ through $m$.
\begin{definition}[the full and cyclic coordinate update operators]
$ $

  For $0<\alpha<1$, we define the full update operator and cyclic coordinate update operator, respectively, as
  \begin{align}
    \label{eq:Talpha}
    T^\alpha & :=I-\alpha S,                                           \\
    \label{eq:Ealpha}
    E^\alpha & :=(I-\alpha S_m)(I-\alpha S_{m-1})\cdots(I-\alpha S_1).
  \end{align}
\end{definition}
 When $(I-S)$ is nonexpansive, the operator $T^\alpha$ has the following properties~\cite[Prop. 4.25]{BauschkeCombettes2011}:
\begin{align}\label{eq:contra}
  \|T^\alpha x-T^\alpha y\|^2\leq\|x-y\|^2-\alpha (1-\alpha) \|Sx-Sy\|^2,\quad \forall x,y\in\cH.
\end{align}
Such an operator is called an \emph{averaged operator} since {we can write it as} $T^\alpha =(1-\alpha) I+\alpha (I-S)$.

For any $x^*\in\zer(S)$, substituting $y=x^*$ in \eqref{eq:contra} and noticing $T^\alpha x^* = x^*-\alpha Sx^*= x^*$ yield the quasi\footnote{The modifier \emph{quasi} is used if the property involves the solution $x^*$.}-contractive property:
\begin{align}\label{eq:qcontra}
  \|T^\alpha x-x^*\|^2\leq\|x-x^*\|^2- \alpha(1-\alpha)\|Sx\|^2,\quad \forall x\in\cH,
\end{align}
which is a key property for the convergence of the KM iteration $x^{k+1}= T^\alpha x^k$. 

The operator $E^\alpha$ characterizes one epoch of Algorithm \ref{alg:cyccu} under the cyclic ordering $1,2,\ldots,m$. Indeed, the iterates $x^k$ of Algorithm \ref{alg:cyccu}  satisfy
\begin{align}\label{eq:Eitr}
  x^{k+1} = E^{\alpha_k} (x^k).
\end{align}
The following proposition follows directly from~\eqref{lipschitz} \eqref{def:L} and~\eqref{eq:Ealpha}.
\begin{proposition}\label{prop:lip}
The operator $E^{\alpha_k}$ is $(1+\alpha_kL)^m$-Lipschitz.
\end{proposition}

We will set $\alpha_k$ so that an inequality similar to \eqref{eq:qcontra}  holds for $E^{\alpha_k}$ and {that} $x^{k}$ (weakly) converges to a point in $\zer(S)$.

\section{Convergence results}\label{sec:analysis}
Our analysis is based on comparing the operator $E^\alpha$ with the operator $T^\alpha$. To simplify notation, let us define the operator
\begin{align}
  R:=\frac{1}{\alpha}(T^{\alpha}-E^{\alpha}).\label{def:R}
\end{align}
By \eqref{eq:Si}, we have the decomposition $S=S_1+S_2+\cdots+S_m=\sum_{i=1}^mS_i$, which  yields
  \begin{align}
    R & =\frac{1}{\alpha}(T^{\alpha}-E^{\alpha})= \frac{1}{\alpha}\big((I-\alpha S) -  (I-\alpha S_m)(I-\alpha S_{m-1})\cdots(I-\alpha S_1)\big) \\
      & = \frac{1}{\alpha}\Big((I-\alpha \sum_{i=1}^mS_i)-\big((I-\alpha S_{m-1})\cdots(I-\alpha S_1)-\alpha S_m (I-\alpha S_{m-1})\cdots(I-\alpha S_1) \big)\Big)\\
    & = \Big( S_m (I-\alpha S_{m-1})\cdots(I-\alpha S_1)- S_m\Big) + \frac{1}{\alpha}\Big((I-\alpha \sum_{i=1}^{m-1}S_i)-(I-\alpha S_{m-1})\cdots(I-\alpha S_1)\Big)\\
    & =  \sum_{i=2}^m \Big(S_i (I-\alpha S_{i-1})(I-\alpha S_{i-2})\cdots(I-\alpha S_1) -S_i\Big) \,,
  \end{align}
  and thus
  \begin{align}
    \|R  x\|^2 & =\sum_{i=2}^m\big\| S_i x -S_i (I-\alpha S_{i-1})(I-\alpha S_{i-2})\cdots(I-\alpha S_1)  x\big\|^2 \,.
  \end{align}
  In addition, we can obtain the following estimate for $R$.

  \begin{lemma} \label{lemma:estR}
    The operator $R$ satisfies the estimate
    \begin{equation}
    \|R x\|\leq \tfrac{\alpha L m}{\sqrt{2} }(1+\alpha L)^{m}\|Sx\|.\label{eq:Rlip}
    \end{equation}
  \end{lemma}
  \begin{proof}
    Let us consider for each $i$,
    \begin{align}
           \Delta_i:=&\| S_i x -S_i (I-\alpha S_{i-1})(I-\alpha S_{i-2})\cdots(I-\alpha S_1)  x \|                                          \\
      \leq & L\|x - (I-\alpha S_{i-1})(I-\alpha S_{i-2})\cdots(I-\alpha S_1)  x\|          .
      \end{align}
      The triangle inequality yields
      \begin{align}
      \Delta_i\leq & L\|x-(I-\alpha S_{i-1})x\|+L\|(I-\alpha S_{i-1})x-(I-\alpha S_{i-1})(I-\alpha S_{i-2})x\|+\cdots                      \\
           & +L\|(I-\alpha S_{i-1})\cdots(I-\alpha S_{2})x-(I-\alpha S_{i-1})\cdots(I-\alpha S_{1})x\|     .
           \end{align}
           Applying Proposition~\ref{prop:lip}, we obtain
          \begin{align}
       \Delta_i\leq & L  \|x-(I-\alpha S_{i-1})x\|+ L(1+ \alpha L) \|  x-(I-\alpha S_{i-2})x\|                                                         \\
           & +  L  \left(1+\alpha L \right)^2 \| x-(I-\alpha S_{i-3})x\| + \dots + L \left(1+\alpha L \right)^{i-2} \| x-(I-\alpha S_{1})x\| \\
      \leq & \alpha L  \|S_{i-1} x\|+ \alpha L(1+ \alpha L) \|  S_{i-2}x\|                                                         \\
           & + \alpha L  \left(1+\alpha L \right)^2 \| S_{i-3} x\| + \dots + \alpha L \left(1+\alpha L \right)^{i-2} \|  S_{1} x\| \\
             \leq & \alpha L(1+ \alpha L)^{m}\sum_{j=1}^{i-1}\|S_jx\|
        \end{align}
        By the Cauchy-Schwarz inequality,        \begin{align}
     \Delta_i \leq & \alpha L(1+ \alpha L)^m\sqrt{i-1}\|Sx\|.
    \end{align}
    Finally, combining the above inequalities yields
    \begin{align}
      \|R x\|^2 & =\sum_{i=2}^m \Delta_i^2 \leq\sum_{i=2}^m  (i-1) \alpha^2 L^2 (1+ \alpha L)^{2m}   \|Sx\|^2 \leq \frac{\alpha^2L^2m^2}{2}(1+\alpha L)^{2m}\|Sx\|^2.
    \end{align}
\hfill     \end{proof}

Pick an arbitrary $x^*\in\zer(S)$. To use the property \eqref{eq:qcontra} of $T^\alpha$, we expand $\| E^{\alpha}x-x^*\|^2$ by $E^{\alpha}=T^{\alpha}-\alpha R$ as follows:
\begin{align}
  \|E^\alpha x-x^*\|^2 & = \|T^{\alpha} x-\alpha R x-x^*\|^2                                              \\
  \label{eq:expd}
                       & = \|T^{\alpha} x-x^*\|^2 - 2\alpha\dotp{T^{\alpha} x-x^*,R x}+\alpha^2\|R x\|^2.
\end{align}
By Young's inequality, the cross term in \eqref{eq:expd} satisfies
\begin{align}\label{eq:crossyoung}
  - 2\alpha\dotp{T^{\alpha} x-x^*,R x}\le  \alpha\eta\|T^{\alpha} x-x^*\|^2+\alpha\eta^{-1}\|Rx\|^2
\end{align}
for any $\eta>0$, which we will set later.
Substituting \eqref{eq:crossyoung} into \eqref{eq:expd} and then applying Lemma \ref{lemma:estR} yield
\begin{align}
  \|E^\alpha x-x^*\|^2 & \le (1+\alpha\eta)  \|T^{\alpha}x-x^*\|^2+\alpha (\eta^{-1}+\alpha)\|Rx\|^2         \\
                        & \le (1+\alpha\eta) \left(  \|T^{\alpha}x-x^*\|^2+ \eta^{-1} \alpha^3 L^2  \tfrac{m^2}{2} (1+\alpha L)^{2m}\|Sx\|^2 \right).\label{eq:ineq2}
\end{align}
Substituting $x=x^k$ and $\alpha=\alpha_k$ and using \eqref{eq:qcontra} yield
  \begin{align}
    \|E^{\alpha_k} x^k-x^*\|^2 & \le (1+\alpha_k \eta) \left(  \|x^k-x^*\|^2- \Big( \alpha_k(1-\alpha_k)  -  \alpha_k^3 L^2  \tfrac{m^2 (1+\alpha_k L)^{2m} }{2\eta} \Big) \|Sx^k\|^2 \right)\label{eq:ineq3}
  \end{align}
  for all $k$.

The existence of the extra coefficient $\alpha_k\eta$  in~\eqref{eq:ineq3} invalidates the traditional analysis. To ensure convergence, we  take two approaches: using slowly decreasing step sizes in Subsection~\ref{sec:dss} and  making stronger assumptions on $S$ in Subsection~\ref{sec:css}.
  \subsection{Slowly decreasing step sizes}\label{sec:dss}
Our convergence will reduce to the analysis of some simple scalar sequences as follows.
\begin{lemma}\label{lm:multerr}
  Consider the sequences
  \begin{align}
    (a_k)_{k\ge 0}, (b_k)_{k\ge 0}, (\xi_k)_{k\ge 0}\subset \{x\in\RR:x \ge 0\},
  \end{align}
  where $\sum_{k\ge0}\xi_k <\infty$ and
  \begin{align}\label{eq:abineq}
    a_{k+1}\le(1+\xi_k)(a_k-b_k).
  \end{align}
  Then, (i) $(a_k)_{k\ge 0}$ is bounded, (ii) there exists $a^*\in\RR_{+}$ such that $\lim_k a_k = a^*\in\RR_+$, and (iii) $\sum_{k\ge 0} b_k <\infty$. 
\end{lemma}

In the simplified case $\xi_k\equiv 0$, the results hold trivially. Indeed, since $a_k\ge 0$ and $a_{k+1}\le a_k$ by \eqref{eq:abineq}, there exists $a^*\in \RR_+$ such that $a_k\to a^*$, and the telescoping sum of \eqref{eq:abineq} yields $\sum_{k\ge 0} b_k <\infty$. 
Lemma \ref{lm:multerr} claims that these results still hold under the multiplicative errors $\xi_k$ that are summable.
\begin{proof}
  Expanding \eqref{eq:abineq} yields
  \begin{align}
    a_{k+1} & \le (1+\xi_k)(1+\xi_{k-1})a_{k-1} - \big((1+\xi_k)b_k+(1+\xi_k)(1+\xi_{k-1})b_{k-1}\big)                       \\
    \label{eq:akbnd}
            & \le \cdots \le \big(\Pi_{j=0}^k(1+\xi_j)\big) a_0 - \sum_{j=0}^k\big(\big(\Pi_{i=j}^k(1+\xi_i)\big)  b_j\big).
  \end{align}
  Let $\xi:=\Pi_{j=0}^k(1+\xi_j)$. Noticing $(1+\xi_j)\leq e^{\xi_j}$ for $\xi_j\ge 0$ and using $\sum_{k\ge 0} \xi_k<\infty$ gives us
  $
  \xi <\infty.
  $
  Applying the inequality $ \sum_{j=0}^k\big(\big(\Pi_{i=j}^k(1+\xi_i)\big)  b_j\big)\ge \sum_{j=0}^k b_j$ to \eqref{eq:akbnd}, we obtain
  \begin{align}
    a_{k+1} & \le {\xi} a_0 - \sum_{j=0}^k b_j,
  \end{align}
  which means that $ a_k\in[0,{\xi} a_0]$ for all $k\ge 0$ and $\sum_{j\ge 0} b_j\le {\xi} a_0 <\infty$. We have proved Parts (i) and (iii).

%
{{To prove Part (ii), let} $d_0:=a_0$ and $d_k:=a_k-a_{k-1},k=1,2,\cdots$. Below we show that $\sum_{k=0}^\infty |d_k|<\infty$.  To this end,
let
\begin{align}
 d_k^+:=\max\{d_k,0\}\quad\text{and}\quad d_k^-:=-\min\{d_k,0\},\quad\forall k,
 \end{align} which yield $d_k^++d_k^-=|d_k|$ and $d_k^+-d_k^-=d_k$. By $a_{k+1}\leq(1+\xi_k)(a_k-b_k)$,  $a_k\leq \xi a_0$, and $b_k\ge 0$, we have $a_{k+1}-a_k\leq \xi_k\xi a_0$, which means $d_{k+1}^+=\max\{d_{k+1},0\}=\max\{a_{k+1}-a_k,0\}\leq\xi_k\xi a_0,\forall k,$ and thus, by summing over $k$, $$\sum_{k=0}^\infty d_k^+\leq a_0+\xi a_0\sum_{k=0}^\infty \xi_k<\infty.$$
  From $0\leq a_k= \sum_{i=0}^k d_k= \sum_{i=0}^k (d_k^+-d_k^-),~\forall k$, we obtain $\sum_{i=0}^k d_k^-\leq\sum_{i=0}^k d_k^+$ and thus
  $\sum_{i=0}^\infty d_k^-<\infty.$  Finally, $$\sum_{k=0}^\infty |d_k|=\sum_{k=0}^\infty d_k^++\sum_{k=0}^\infty d_k^-<\infty.$$Hence, $\{a_k\}_{k\ge 0}$ is a Cauchy sequence, and Part (ii) holds.}\hfill
\end{proof}
{\begin{lemma}\label{lm:smbconv}
Let $(a_k)_{k\ge 0}$  be a nonnegative sequence with the following properties: 
  \begin{enumerate}
  \item\label{smb} $
  \sum_{k=0}^\infty\frac{1}{\sqrt{k}}a_k<+\infty,
  $
  and 
  \item  there exists some $B>0$ such that for all $k\geq 1$,
    $|a_{k+1}-a_k|\le Bk^{-\frac{1}{2}}$.
\end{enumerate}   
  Then, we have $\lim_{k\to\infty} a_k=0$.
\end{lemma}
\begin{proof}
We apply proof by contradiction and 
assume a constant $C>0$ such that $\limsup_{k\to\infty}a_k\ge C$. For each $N>0$, there exists $n\geq N$ such that $F:=a_n\geq\frac{C}{2}$.
By Part 2, $a_{k+1}-a_k\geq-B k^{-\frac{1}{2}}$, so we have $a_{n+i}\geq\frac{F}{2}$ for all $0\leq i\leq\lfloor n'\rfloor$, where $n'=\tfrac{F\sqrt{n}}{2B}$.
Then we obtain
\begin{align}
\sum_{i=n}^{n+\lfloor n'\rfloor}i^{-1/2}a_i&\geq\frac{F}{2}\sum_{i=n}^{n+\lfloor n'\rfloor}i^{-1/2}\geq\frac{F}{2}\int_n^{n+n'}x^{-1/2}dx\\
&\geq F\cdot(\sqrt{n+n'}-\sqrt{n})=\tfrac{F^2\sqrt{n}}{2B(\sqrt{n+n'}+\sqrt{n})}.
\end{align}
For any sufficiently large $n$, we have $n'=\tfrac{F\sqrt{n}}{2B}\leq 3n$, so the last term above is at least $\tfrac{F^2}{6B}\geq\frac{C^2}{24B}$. Therefore, by Cauchy's criterion, the sequence $\frac{1}{\sqrt{k}}a_k$ is not summable, which  contradicts Part~1. Hence, we  must have $\limsup_{k\to\infty} a_k=0$. Since all $a_k\ge 0$, we have $\lim_{k\to\infty} a_k=0$.
\end{proof}}

Applying 
 Lemma \ref{lm:multerr} {and \ref{lm:smbconv}}, we can establish the following convergence result.
\begin{theorem}\label{thm:decreasingalpha2}
Let Algorithm \ref{alg:cyccu} use the step size sequence
\begin{align}\label{eq:alphadim3}
  \alpha_k = \frac{1}{k^{1/2}}.
\end{align}
Under Assumption \ref{assump:S}, $x^k$ (weakly) converges to $x^*$ for some $x^*\in\zer(S)$.
\end{theorem}
\begin{proof}
We choose $\eta=m^2 L^2\alpha_k^2(1+\alpha_kL)^{2m}$ in~\eqref{eq:ineq3} to get
\begin{align}
  \|x^{k+1}-x^*\|^2 \le
  (1+m^2 L^2\alpha_k^3(1+\alpha_kL)^{2m})\Big(\|x^k-x^*\|^2 -
  \alpha_k \left(\tfrac{1}{2}-\alpha_k  \right) \|Sx^k\|^2 \Big) \,.
\end{align}
With \eqref{eq:alphadim3} and $(1+\alpha_kL)^{2m}\leq (1+L)^{2m}$, we have $\sum_k m^2 L^2\alpha_k^3(1+\alpha_kL)^{2m}<\infty$, and $\alpha_k (\frac{1}{2} - \alpha_k )\ge 0$ for $k\ge 4$. Since $E^{\alpha_k}$ is $(1+\alpha_kL)^m$-Lipschitz and $\alpha_k\le 1$, $\|x^{4}-x^*\|\le (1+L)^{4} \|x^{0}-x^*\|$.

  By Lemma \ref{lm:multerr} (applied to $k\ge 4$), $\|x^{k+1}-x^*\|$ converges to some $c\ge 0$, and thus $(x^k)_{k\ge 0}$ is bounded and has a weak cluster point $\bar{x}$, and
  \begin{align}
    \sum_{k\ge 4} \alpha_k \bigg(\frac{1}{2}-\alpha_k\bigg)\|Sx^k\|^2 <\infty.\label{eq:summable0}
  \end{align}
Since $\frac{1}{4\sqrt{k}}\leq\alpha_k \big(\frac{1}{2}-\alpha_k\big)\leq\frac{1}{2\sqrt{k}}$ for $k\geq 16$, we obtain from \eqref{eq:summable0}:
\begin{equation}
  \sum_{k\ge 4} \frac{1}{\sqrt{k}}\|Sx^k\|^2 <\infty.\label{eq:summable}
  \end{equation}
%
By Lemma \ref{lm:multerr}, there exists some $A\in\RR$ such that $\|x_k-x^*\|\leq A$. We have
\begin{align}
\|x^{k+1}-x^k\|&=\|E^{\alpha_k}x^k-x^k\|=\|T^{\alpha_k}x^k-\alpha_kRx^k-x^k\|
=\|-\alpha_kSx^k-\alpha_kRx^k\| .
\end{align}
By $Sx^*=0$, the triangle inequality, and~\eqref{eq:Rlip},
\begin{align}
\|x^{k+1}-x^k\|&={\|-\alpha_kSx^k+\alpha_kSx^*-\alpha_kRx^k\|}\\
&\leq \alpha_k\|Sx^k-Sx^*\|+\alpha_k\|Rx^k\|\\
&\leq\big(\alpha_k+\tfrac{\alpha_k^2 L m}{\sqrt{2} }(1+\alpha_k L)^{m}\big)\|Sx^k-Sx^*\|.
\end{align}
Because $S$ is 2-Lipschitz, $\alpha_k\le 1$, and $\|x_k-x^*\|\leq A$, the above inequality yields
\begin{align}
\|x^{k+1}-x^k\|&\leq 2\alpha_k\big(1+\tfrac{\alpha_k L m}{\sqrt{2} }(1+\alpha_k L)^{m}\big)\|x^k-x^*\|\\
&\leq 2\alpha_k\big(1+\tfrac{ L m}{\sqrt{2} }(1+ L)^{m}\big)A\\
&\leq 2\big(1+\tfrac{ L m}{\sqrt{2} }(1+ L)^{m}\big)A \cdot k^{-\frac{1}{2}}.
\end{align}
Then we have
\begin{align}
\left|\|Sx^{k+1}\|^2-\|Sx^k\|^2\right|&=\left|\|Sx^{k+1}\|-\|Sx^k\|\right|(\|Sx^{k+1}\|+\|Sx^k\|)\\
&\leq \|Sx^{k+1}-Sx^k\|(\|Sx^{k+1}-Sx^*\|+\|Sx^k-Sx^*\|)\\
&\leq 2\|x^{k+1}-x^k\|(2\|x^{k+1}-x^*\|+2\|x^k-x^*\|)\\
&\leq \underbrace{2\cdot 2\big(1+\tfrac{ L m}{\sqrt{2} }(1+ L)^{m}\big)A\cdot 4A}_{=:B}\cdot k^{-\frac{1}{2}} = B k^{-\frac{1}{2}} .
\end{align}
{Hence, $(\|Sx^k\|^2)_{k\geq 0}$ satisfies the two conditions of Lemma~\ref{lm:smbconv}, so $\lim_{k\to\infty}\|Sx^k\|^2=0$.}

{Finally,} we adapt the convergence proof of KM iteration \cite{krasnosel1955two} to our setting. 

  Recall the demicloseness principle (cf. textbook \cite{BauschkeCombettes2011}): if $T$ is nonexpansive, $z^j \rightharpoonup \bar{z}$, and $\lim\|(I-T)z^j\|=0$, then $\bar{z}=T(\bar{z})$. Applying this principle to $T=I-S$ and the subsequence of $(x^k)_{k\ge 0}$ that weakly converges to $\bar{x}$, we obtain $\bar{x}\in\zer(S)$.

  Next, we show that any weak cluster point $\bar{y}$ of $(x^k)_{k\ge 0}$ must equal $\bar{x}$. The demicloseness principle again yields $\bar{y}\in\zer(S)$. Substituting $x^*=\bar{x}$ and then $x^*=\bar{y}$ and following the argument above yield the limits $\lim_{k}\|x^{k+1}-\bar{x}\|=c_x$ and $\lim_{k}\|x^{k+1}-\bar{y}\|=c_y$. Algebraically,
  \begin{align}\label{eq:5term}
    2\dotp{x^k,\bar{x}-\bar{y}} = \|x^k-\bar{x}\|^2 -\|x^k -\bar{y}\|^2 + \|\bar{x}\|^2-\|\bar{y}\|^2,
  \end{align}
  whose right-hand side converges to the constant $c':=c_x^2-c_y^2+\|\bar{x}\|^2-\|\bar{y}\|^2$ as $k\to\infty$. Passing the limits of \eqref{eq:5term} over the two subsequences that weakly converge to $\bar{x}$ and to $\bar{y}$, respectively, yields
  $2\dotp{\bar{x},\bar{x}-\bar{y}}=2\dotp{\bar{y},\bar{x}-\bar{y}}=c'.$
  Hence, $\|\bar{x}-\bar{y}\|^2=0$ and $(x^k)_{k\ge 0}$
  weakly converges to $\bar{x}\in\zer(S)$. \hfill \end{proof}

  It follows immediately from \cite[Lemma 3]{davis2014convergence} that $\min_{j\le k}\big\{ \frac{1}{\sqrt{j}}\|Sx^j\|^2\big\} = {o\big( \frac{1}{k} \big)}.$ Therefore, we have the following Corollary.
  \begin{corollary}
    Under the setting of Theorem \ref{thm:decreasingalpha2}, the reduction rate of running-minimal residual is
    \begin{align}
    \min_{j\le k}\{\|Sx^j\|^2\} = {o\big(\frac{1}{\sqrt{k}}\big)}.
    \end{align}
  \end{corollary}
Note that we do not write $\min_{j\le k}\{\|Sx^j\|\}=o(1/k^{1/4})$ since $\|Sx^k\|^2$ naturally appears in our analysis.

In the current setting, we cannot expect to have a convergence rate {for} $\|x^k - x^*\|$. 

  \subsection{Fixed step size}\label{sec:css}
  In Theorem \ref{thm:decreasingalpha2}, $\alpha_k$ in \eqref{eq:alphadim3} are decreasing. 
Next, we study convergence under the fixed step size $\alpha_k\equiv \alpha<1$. We need an additional assumption as follows:
  \begin{assumption}[quasi-$\mu$-strong monotonicity]\label{assump:qstrong}
  There exists some $\mu>0$ such that the operator $S$ satisfies \begin{align}\label{eq:qstrong}
      \langle Sx,x-x^*\rangle \ge \mu \|x-x^*\|^2,\quad\forall x^*\in\zer (S),x\in\cH.
    \end{align}
  \end{assumption}
  This assumption {is weaker than assuming $S$ be strongly monotone. In addition, our assumption} ensures that $x^*$ is the sole element of $\zer(S)$. Indeed, any $\bar{x}\in\zer(S)$ must obey $0=\langle S\bar{x},\bar{x}-x^*\rangle\ge \mu^2\|\bar{x}-x^*\|^2$.

  {By the Cauchy-Schwarz inequality}, Assumption~\ref{assump:qstrong} also implies $\|Sx\|\geq\mu\|x-x^*\|$. Note that $\|Sx\|=\|Sx-Sx^*\|\leq 2\|x-x^*\|$, so we at least know $\mu\leq 2$.
\begin{theorem}\label{thm:constantalpha}
Under Assumptions \ref{assump:S} and~\ref{assump:qstrong}, using the fixed step size
  \begin{align}\label{eq:alphacond}
     \alpha=\min\left\{\frac{1}{4mL},\frac{\mu}{4\sqrt{2}mL},\frac{2mL}{17mL+2\mu^2}\right\}
    \end{align}
    (where the constants naturally appear in the proof below),
Algorithm \ref{alg:cyccu}  generates a sequence $x^k$ that converges to $ x^*\in\zer(S)$.
We further have $$\|x^k-x^*\|^2\leq \rho^k\|x^0-x^*\|^2$$ with $\rho = 1 -  \frac{ \alpha \mu^2 }{2}< 1$. When $m$ and $\frac{L}{\mu}$ are large, we have $\alpha=O(\frac{\mu}{mL})$.
  \end{theorem}
  {\begin{remark}
  It difficult to remove the $\frac{1}{m}$ factor from the step size $\alpha$ in Theorem~\ref{thm:constantalpha}. Although it is quite small, we are able to achieve the same $[1-O(\frac{1}{m})]^k$ convergence rate as established for the block coordinate (proximal) gradient descent method~\cite{hong2013iteration,beck2013convergence}, which solves a minimization problem and is a special case of our Algorithm~\ref{alg:cyccu}. In their case, an objective function can be used to analyze convergence, but it still seems difficult to remove the $\frac{1}{m}$ factor. Therefore, it will be more so in our setting, which does not use any objective function and solely relies on iterate error. 
  \end{remark}}
   \begin{proof}
Combining~\eqref{eq:qcontra} and~\eqref{eq:ineq2} with $x=x^k$ and using $x^{k+1}=E^{\alpha} x^k$ , we obtain\\[-20pt]
\begin{align}
  \| x^{k+1}-x^*\|^2
  \le (1+\alpha \eta) \left(  \|x^k-x^*\|^2- \big( \alpha (1-\alpha)  - \eta^{-1} \alpha^3 L^2  \tfrac{m^2 (1+\alpha L)^{2m} }{2} \big) \|Sx^k\|^2  \right).
\end{align}
We use \eqref{eq:qstrong} to get that
\begin{align}
  \| x^{k+1}-x^*\|^2
&  \le  \|x^k-x^*\|^2  +\frac{\alpha \eta}{\mu^2} \|Sx^k\|^2  \\
&- (1+\alpha \eta)  \left( \alpha (1-\alpha)  - \eta^{-1} \alpha^3 L^2  \tfrac{m^2 (1+\alpha L)^{2m} }{2} \right) \|Sx^k\|^2   .
\end{align}
If we can ensure
\begin{align}
 \frac{1}{2}  +\frac{ \eta}{\mu^2}  -(1+\alpha \eta)  \left(  1-\alpha   - \eta^{-1} \alpha^2 L^2  \tfrac{m^2 (1+\alpha L)^{2m} }{2} \right)   \leq0 ,\label{ineq:alpha}
\end{align}
we will have by Assumption~\ref{assump:qstrong}:
\begin{align}
    \|x^{k+1}-x^*\|^2 &  \le  \|x^{k} - x^*\|^2 - \frac{\alpha}{2} \|Sx^k\|^2  \le \left(1 -  \tfrac{ \alpha \mu^2 }{2} \right) \|x^{k} - x^*\|^2.
\end{align}
Introduce $\beta$ to rewrite $\alpha = \frac{ \beta  }{ 2 m L} $. Let $\eta = \frac{\mu^2}{4} $. By $ \left(1+ \frac{\beta}{2 m} \right)^{2m}  < e^{\beta} $, we have~\eqref{ineq:alpha} provided that
\begin{align}
\tfrac{3}{4}- \left(1+ \tfrac{ \beta  \mu^2 }{ 8mL} \right)  \left(  1- \tfrac{ \beta   }{ 2 m L}  -\tfrac{ \beta^2e^\beta }{ 2\mu^2}  \right)   \leq 0 \,.\label{ineq:beta}
\end{align}
By simplification,~\eqref{ineq:beta} is equivalent to
\begin{align}
\tfrac{\beta^3e^\beta}{16mL}+\beta^2\big(\tfrac{e^\beta}{2\mu^2}+\tfrac{\mu^2}{8m^2L^2}\big)+\beta\big(\tfrac{1}{2mL}-\tfrac{\mu^2}{8mL}\big)\leq \tfrac{1}{4},
\end{align}
When $\beta\leq\frac{1}{2}$, we have $e^\beta<2$ and $e^\beta\beta<1$. Hence, we only need to ensure
\begin{align}
\tfrac{\beta}{32mL}+\big(\tfrac{\beta^2}{\mu^2}+\tfrac{\mu^2\beta}{16m^2L^2}\big)+\beta\big(\tfrac{1}{2mL}-\tfrac{\mu^2}{8mL}\big)\leq \tfrac{1}{4},
\end{align}
which can be guaranteed by
\begin{align}
\tfrac{\beta^2}{\mu^2}\leq\tfrac{1}{8}\text{ and }\beta\big(\tfrac{1}{32mL}+\tfrac{1}{2mL       }+\tfrac{\mu^2}{16m^2L^2}\big)\leq\tfrac{1}{8}.
\end{align}
Therefore we need
\begin{align}
\beta\leq\tfrac{\mu}{2\sqrt{2}}\text{  and }\beta\leq\tfrac{1}{\frac{17}{4mL}+\frac{\mu^2}{2m^2L^2}}.
\end{align}
Therefore, 
we arrive at \eqref{eq:alphacond}.
 \hfill \end{proof}
\subsection{Primal-dual coordinate update and its convergence metric}\label{sec:pd}

In this section we briefly review primal-dual algorithms and adapt Algorithm~\ref{alg:cyccu} and its analysis in Subsections~\ref{sec:dss} and~\ref{sec:css} {to obtain faster coordinate-update algorithms}.

Primal-dual algorithms~\cite{chambolle2011first,condat2013primal,vu2013splitting} can solve the following problem:
\begin{equation}
\Min_{x\in\cH} ~g(x)+h(x)+f(Ax),\label{pdproblem}
\end{equation}
where $g$ is a differentiable convex function; $f,h:\cH\to\RR\cup\{\infty\}$ are extended-value convex functions, which are not necessarily differentiable, and $A:\cH\to\cG$ is a linear operator from $\cH$ to another Hilbert space $\cG$. Through indicator functions, $f,h$ can model constraints like $x\in \cC$ or $Ax\in \cC$, where $\cC$ is some closed convex set. Primal-dual algorithms~involve a dual variable $s\in\cG$  and iteratively update both $x$ and $s$ by the iteration:
\begin{equation}
\left\{
\begin{array}{l}
x^{k+1}=\prox_{\eta h}(x^k-\eta(\nabla g(x^k)+A^\top s^k)),\\
s^{k+1}=\prox_{\gamma f^*} (s^k+\gamma A(2x^{k+1}-x^k)),
\end{array}
\right.\label{vucondat2}
\end{equation}
where $f^*$ is the Fenchel dual of $f$ and $\prox_{\eta h}$ is the proximal operator of the function $h$ defined as
\begin{align}
\prox_{\eta h}(x):=\argmin_{y\in\cH}\, h(y)+\frac{1}{2\eta}\|y-x\|^2.
\end{align}
Define $z:={x\choose s}$ and rewrite~\eqref{vucondat2} as $z^{k+1}=Tz^k$. It is shown~\cite{combettes2014forward} that (with proper choice of step sizes) $T$ is nonexpansive under the metric induced by the norm $\|z\|_M=\sqrt{\langle z,Mz\rangle}$, where $M$ is a certain positive definite linear operator.

To describe our coordinate-update algorithm, we assume that a product-form Hilbert space $\cG=\cG_1\times\cG_2\times\cdots \times \cG_p$ and break the dual variable $s\in\cG$ into $p$ blocks: $$s=(s_1,\,s_2,\,\ldots,s_p)$$ with $s_i\in\cG_i,~ i=1,2,...,p$. Each step of Algorithm~\ref{alg:cyccu} picks a coordinate $z_i$ of $z$, which can be a coordinate of either $x$ or $s$. {Then it} follows~\eqref{vucondat2} to updated $z_i$ only. We use the techniques in \cite[Section 4]{peng2016coordinate} to ensure such coordinate updates computationally worthy. That is, updating one coordinate of $x$ or $s$ only takes $O(\frac{1}{m+p})$ of the cost of computing the full update~\eqref{vucondat2}.

When  $g=0$ and $h=0$ (which is the case of the test problems presented in Section~\ref{sec:numerical}), $\prox_{\eta h}$ is the identity operator and, thus, we can {eliminate $x^{k+1}$ in} the $s$ update {of~\eqref{vucondat2}} and obtain the simpler iteration:
\begin{equation}
\left\{
\begin{array}{l}
x^{k+1}=x^k-\eta A^\top s^k,\\
s^{k+1}=\prox_{\gamma f^*} (s^k+\gamma Ax^{k}-2\gamma\eta AA^\top s^k),
\end{array}
\right.\label{pvucondat}
\end{equation}
where $s^{k+1}$ can be computed from $x^k$ and $s^k$.

Computing coordinate updates requires the \emph{caching} technique. In particular, for the coordinate update based on~\eqref{pvucondat}, we cache the variable $Ax^k$ when $m\gg p$. To compute $s_i^{k+1}$, we directly use $(Ax^k)_i$ instead of multiplying $A_{i,:}$ and $x^k$. When $x_i^k$ is updated to $x_i^{k+1}$, we update $Ax^k$ to $Ax^{k+1}$ by $Ax^{k+1}=Ax^k+A_{:,i}\cdot(x^{k+1}_i-x^k_i)$, which takes only $O(p)$ operations. This is cheaper than computing $Ax$ directly, which takes $O(m)$ numbers of operations. 

Our proofs in Sections~\ref{sec:dss} and~\ref{sec:css} apply to primal-dual coordinate-update algorithms after adjusting certain constants for $\|\cdot\|_M$ by the next lemma.
\begin{lemma}
Let $\lambda_{\max},\lambda_{\min}$ be the largest and smallest eigenvalues of $M$, respectively, and $\kappa:=\frac{\lambda_{\max}}{\lambda_{\min}}$ be its condition number. We have for any $z\in\cH$,
\begin{equation}
\frac{1}{\kappa^2}\|z\|^2_M\leq\sum_{i=1}^m\|z_i\|^2_M\leq\kappa^2\|z\|^2_M.
\end{equation}
\end{lemma}
\begin{proof}
\begin{align}
\sum_{i=1}^m\|z_i\|^2_M\leq&\sum_{i=1}^m\lambda^2_{\max}\|z_i\|^2=\lambda^2_{\max}\|z\|^2\leq\frac{\lambda^2_{\max}}{\lambda^2_{\min}}\|z\|^2_M=\kappa^2\|z\|^2_M.
\end{align}
The other half is similar.
\hfill\end{proof}

In particular, when $S$ is 1/2-cocoercive under the norm $\|\cdot\|_M$, we have $\|S_ix-S_iy\|_M\leq 2\kappa\|x-y\|_M$ for all $i$. {Therefore, it is easy to extend Lemma \ref{lemma:estR} as follows:}
\begin{lemma}\label{Rlip2}
Let $M$ be a symmetric positive definite matrix with condition number $\kappa$, $I-S$ be nonexpansive under the norm $\|\cdot\|_M$, $T^\alpha,E^\alpha$ and $R$ be defined as in~\eqref{eq:Talpha}, \eqref{eq:Ealpha} and~\eqref{def:R}, {respectively}, and $L$ be defined as~\eqref{def:L}. The operator $R$ satisfies the estimate
    \begin{equation}
    \|R x\|\leq \frac{\alpha L m\kappa^2}{\sqrt{2} }(1+\alpha L)^{m}\|Sx\|.\label{eq:Rlipscaled}
    \end{equation}
\end{lemma}

Based on Lemma~\ref{Rlip2}, the statement of Theorem~\ref{thm:decreasingalpha2} still holds, and the proof is similar except that the choice of $\eta$ is adjusted.
Theorem~\ref{thm:constantalpha}  still holds, too, but with the step size $\alpha=O(\frac{\mu}{\kappa mL})$ and  some new constants in its proof.
\section{Numerical experiments}\label{sec:numerical}
In this section we illustrate the efficiency of Algorithm~\ref{alg:cyccu} on three different applications: $\ell_1$ based robust linear regression, computed tomography and nonnegative matrix factorization. They are important problems in statistics, medical imaging, and machine learning, respectively. The first two problems cannot be solved by the traditional coordinate descent algorithms, and the last one is a nonconvex problem.

We use these results to illustrate the following two points:
\begin{itemize}
  \item In spite of the small theoretical step sizes, practical problems in our preliminary experiments accept very large step sizes, which contribute to the great performance of our Algorithm \ref{alg:cyccu}.
  \item Our Algorithm \ref{alg:cyccu} is significantly faster than the standard fixed-point iteration, which performs the full update in each iteration, and also faster than the algorithm using randomized coordinate selection.
\end{itemize}
In all of our numerical experiments, convergence was observed with $\alpha_k=1$ in Algorithm~\ref{alg:cyccu}. In addition, Algorithm \ref{alg:cyccu} and its randomized variant both admit larger \emph{intrinsic} step sizes, which are $\eta,\gamma$ in~\eqref{pvucondat}. This brings a significant speed advantage to the coordinate update algorithms over the standard fixed-point iteration. 

Our numerical experiments are coded in Matlab that is  running on a laptop with 2.7 GHz Intel Core
i5 and 8 Gigabytes of RAM. 
\subsection{$\ell_1$ based robust linear regression}
Consider the problem:
\begin{align}
\Min_{x\in\RR^m} \,f(Ax):=\|Ax-b\|_1,\label{rls}
\end{align}
where $A\in\RR^{n\times m}$ and $b\in\RR^{n}$ are given. We apply Algorithm \ref{alg:cyccu} with diagonal scaling~\cite{pock2011diagonal} to solve~\eqref{rls}. Specifically, the fixed-point iteration is:
\begin{subequations}\label{rls:iter}
\begin{align}
x^{k+1}&=x^k-H A^\top s^k,\label{rls:x}\\
s^{k+1}&=\prox_{\Gamma f^*} (s^k+\Gamma A(2x^{k+1}-x^k)),\label{rls:s}
\end{align}
\end{subequations}
where $f^*(y)=\iota_{\|\cdot\|_\infty\leq 1}(y)+y^\top b$ \footnote{$\iota_C$ is the indicator function: $\iota_C(x)=0$ if $x\in C$ and $=\infty$ if $x\not\in C$.} and $H,\Gamma$ are diagonally scaling matrices with $H_{ii}=\frac{1}{\|A_{:,i}\|_1},\Gamma_{ii}=\frac{1}{\|A_{i,:}\|_1}$.
Furthermore, we have
\begin{align}
\prox_{\Gamma f^*}(y):=\argmin_t f^*(t)+\frac{1}{2}\|t-y\|^2_{\Gamma^{-1}}=\Proj_{\|\cdot\|_\infty\leq 1}(y-\Gamma b).\label{eq:proxf}
\end{align}
Here, $y=\Proj_{\|\cdot\|_\infty\leq 1}(x)$ can be computed component-wise as $y_i=\Proj_{[-1,1]}(x_i)=\max\{-1,\min\{1,x_i\}\}$, $i=1,\ldots,n$.
Substituting the $x$ update~\eqref{rls:x} into the $s$ update~\eqref{rls:s} and using~\eqref{eq:proxf}, we can rewrite~\eqref{rls:iter} as
\begin{subequations}\label{rls:iter2}
\noeqref{rls:x2,rls:s2}
\begin{align}
x^{k+1}&=x^k-H A^\top s^k,\label{rls:x2}\\
s^{k+1}&=\Proj_{\|\cdot\|_\infty\leq 1} (s^k-\Gamma b+\Gamma A(x^{k}-2H A^\top s^k)),\label{rls:s2}
\end{align}
\end{subequations}
which is more suitable for coordinate update.

We define $z:=\begin{pmatrix}
x\\
s\\
\end{pmatrix}$ and write \eqref{rls:iter2} as $z^{k+1}:=Tz^k$. To this operator $T$ and $S=I-T$ do we apply Algorithm~\ref{alg:cyccu}. 


In our experiments, we let $n=500$ and $m=100$. The elements of $A$ and $b$ are sampled from the standard normal distribution. The solution $x^*$ and the optimal function value $f^*$ are obtained by solving Problem~\eqref{rls} using CVX. The starting point $z^0={x^0\choose s^0}$ is set to be $0$ . The block size is  1.
\begin{figure}[!htb]\centering
\includegraphics[width=0.8\linewidth]{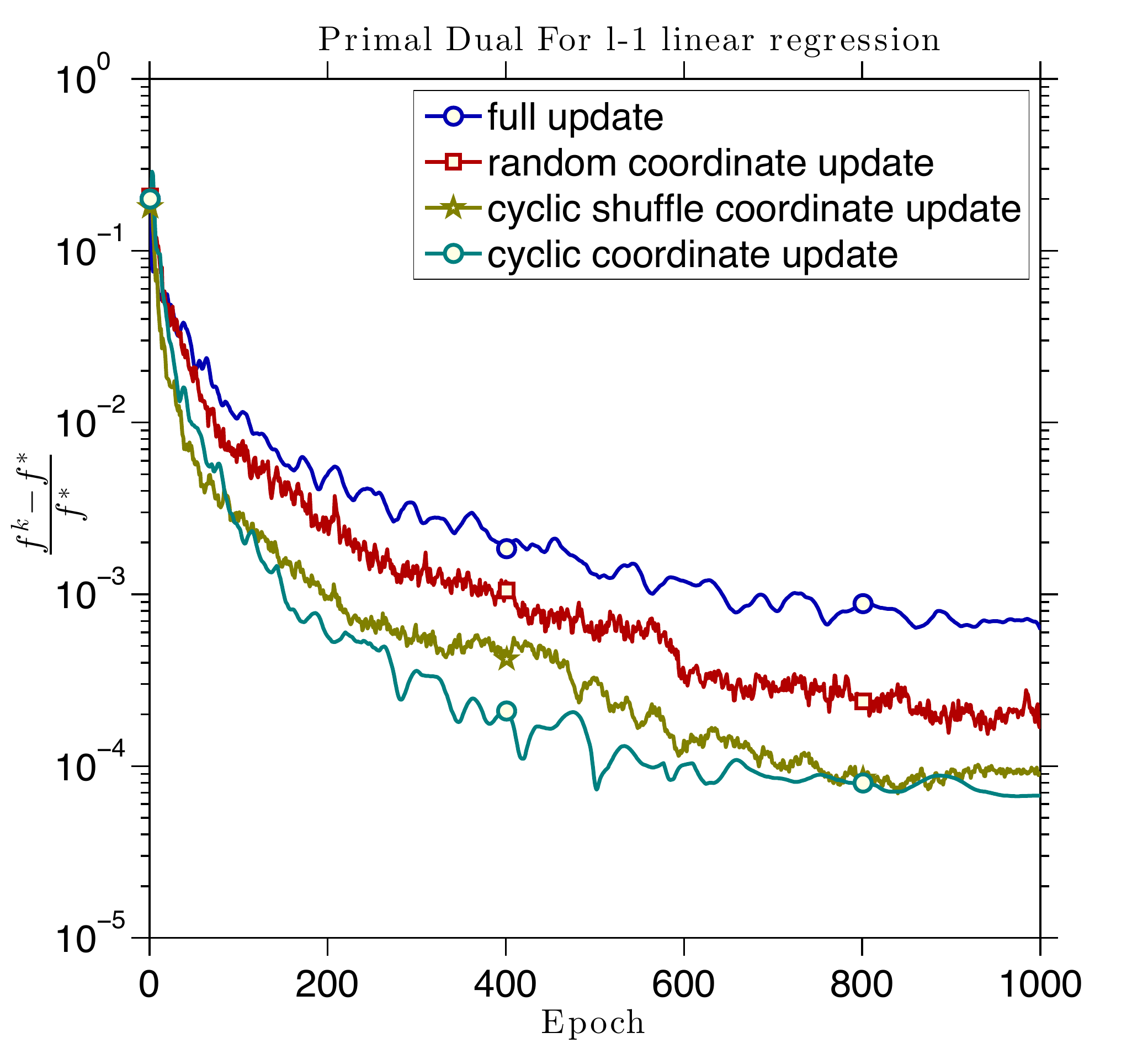}
\caption{$\ell_1$ based robust least squares}
\label{fig:l1}
\end{figure}

Since Problem~\eqref{rls} may have multiple solutions, we use $\frac{f^k-f^*}{f^*}$ to measure convergence. To fully explore the power of each algorithm, we multiply both step size matrices $\Gamma$ and $H$ in~\eqref{rls:iter2} by a scaling factor $\nu$. 
We set $\nu=6$ for the full update (a larger value leads to divergence) and $\nu=12$ for the coordinate updates. Figure~\ref{fig:l1} plots $\frac{f^k-f^*}{f^*}$ versus epoch. We can see that cyclic and shuffled cyclic algorithms perform better than the random algorithm, which is further faster than the full update algorithm.
\subsection{Computed tomography (CT)}
Consider the image recovery problem:
\begin{equation}
\Min_{x\in\RR^m}~\lambda \|\nabla x\|_1+\frac{1}{2}\|Ax-b\|^2,\label{eq:CT}
\end{equation}
where $x\in\RR^m$ is the unknown two-dimensional image  (reformulated as a vector), $\nabla\in\RR^{n_1\times m}$ is the finite difference operator, $\lambda$ is a scaling factor, $A\in\RR^{n_2\times m}$ is the Radon transform matrix, and $b\in\RR^{n_2}$ is the observed CT data, which is contaminated by random noise. Let $n:=n_1+n_2$.

By defining
\begin{equation}
B:={\nabla \choose A},\qquad f(p,q):=\lambda\|p\|_1+\frac{1}{2}\|q-b\|^2, \text{ for } p\in\RR^{n_1},q\in\RR^{n_2},
\end{equation}
we can rewrite ~\eqref{eq:CT}  as $$\Min_{x\in\RR^m} f(Bx).$$
To solve this problem, we apply  Algorithm~\ref{alg:cyccu} to the fixed-point iteration (see \cite[section 5.2.2]{peng2016coordinate} for its derivation):
%
\begin{subequations}\label{eqn:pd_tvl2}
\begin{align}
x^{k + 1} &= x^k - \eta (\nabla^\top s^k + A^\top  t^k), \\
s^{k + 1} &= \Proj_{\|\cdot\|_{\infty} \leq \lambda} \left(s^k + \gamma \nabla (x^k - 2\eta (\nabla^\top s^k + A^\top t^k))\right), \\
t^{k+1} &= \frac{1}{1 + \gamma} \left(t^k + \gamma A (x^k - 2 \eta (\nabla^\top s^k + A^\top t^k)) - \gamma b \right).
\end{align}
\end{subequations}
%
We implement Algorithm~\ref{alg:cyccu} with $\alpha_k=1$ 
and compare it to the full update and random coordinate selection.

We generate a thorax phantom of size $284\times 284$. The Radon matrix $A$ is generated by Siddon's algorithm~\cite{siddon1985fast}. 

The image $x$ is partitioned into $284$ blocks, with each block corresponding to a column of the image. The dual variables $s, t$ are also partitioned into 284 blocks accordingly. A
block of $x$ and the corresponding blocks of $s$ and $t$ are bundled together as
a single block. In each iteration, a bundled block of $x,s,t$ is chosen and updated.

The step sizes $\eta$ and $\gamma$ are hand tuned for both the full and coordinate updates. The different rules of coordinate selection use the same step size.
\begin{figure}[!htb]
\begin{subfigure}{0.5\linewidth}
        \centering
        \includegraphics[width=1\linewidth]{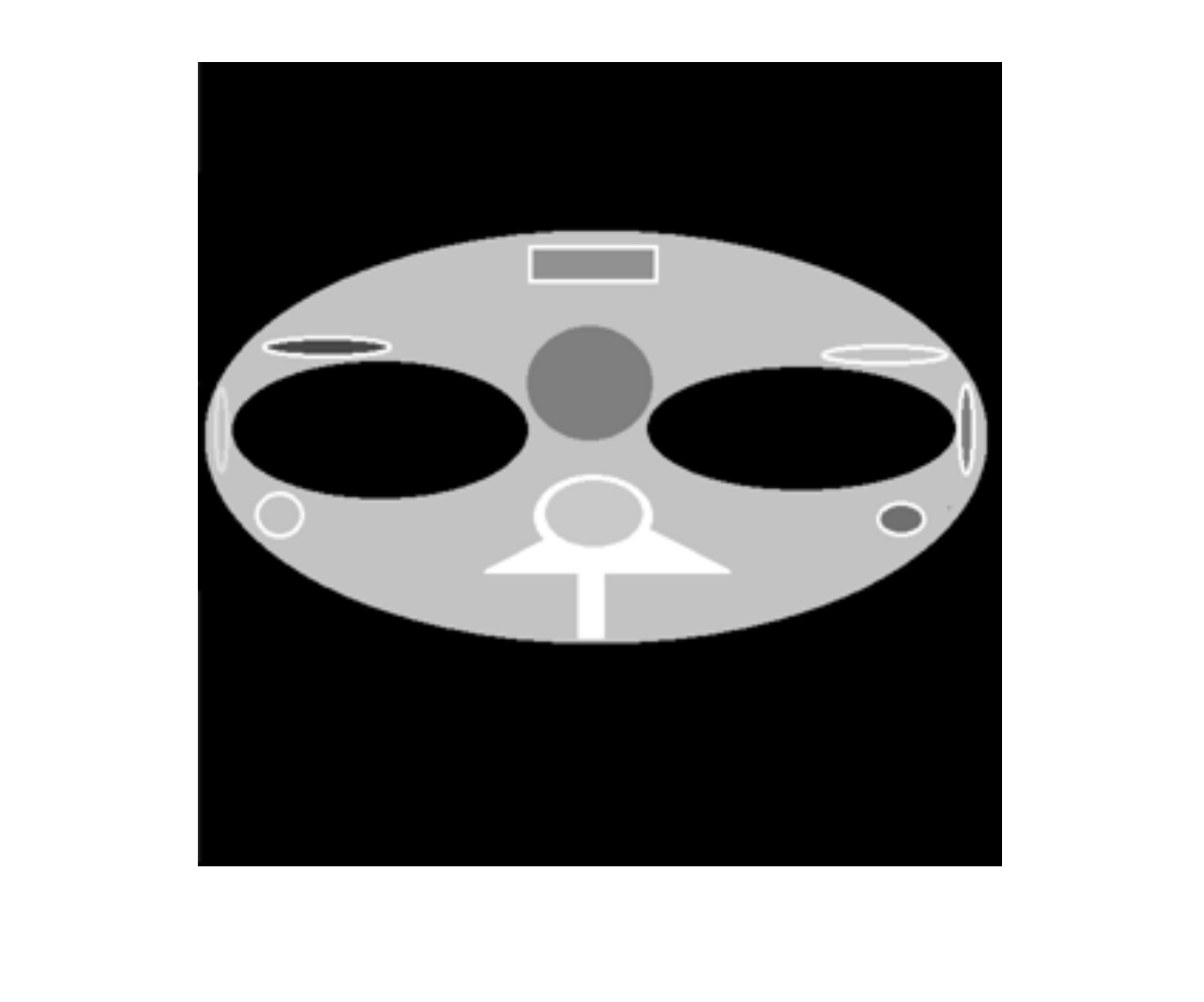}
         \caption{Phantom true image}\label{fig:pds_a}
    \end{subfigure} %
    \quad
    \begin{subfigure}{0.5\linewidth}
        \centering
        \includegraphics[width=1\linewidth]{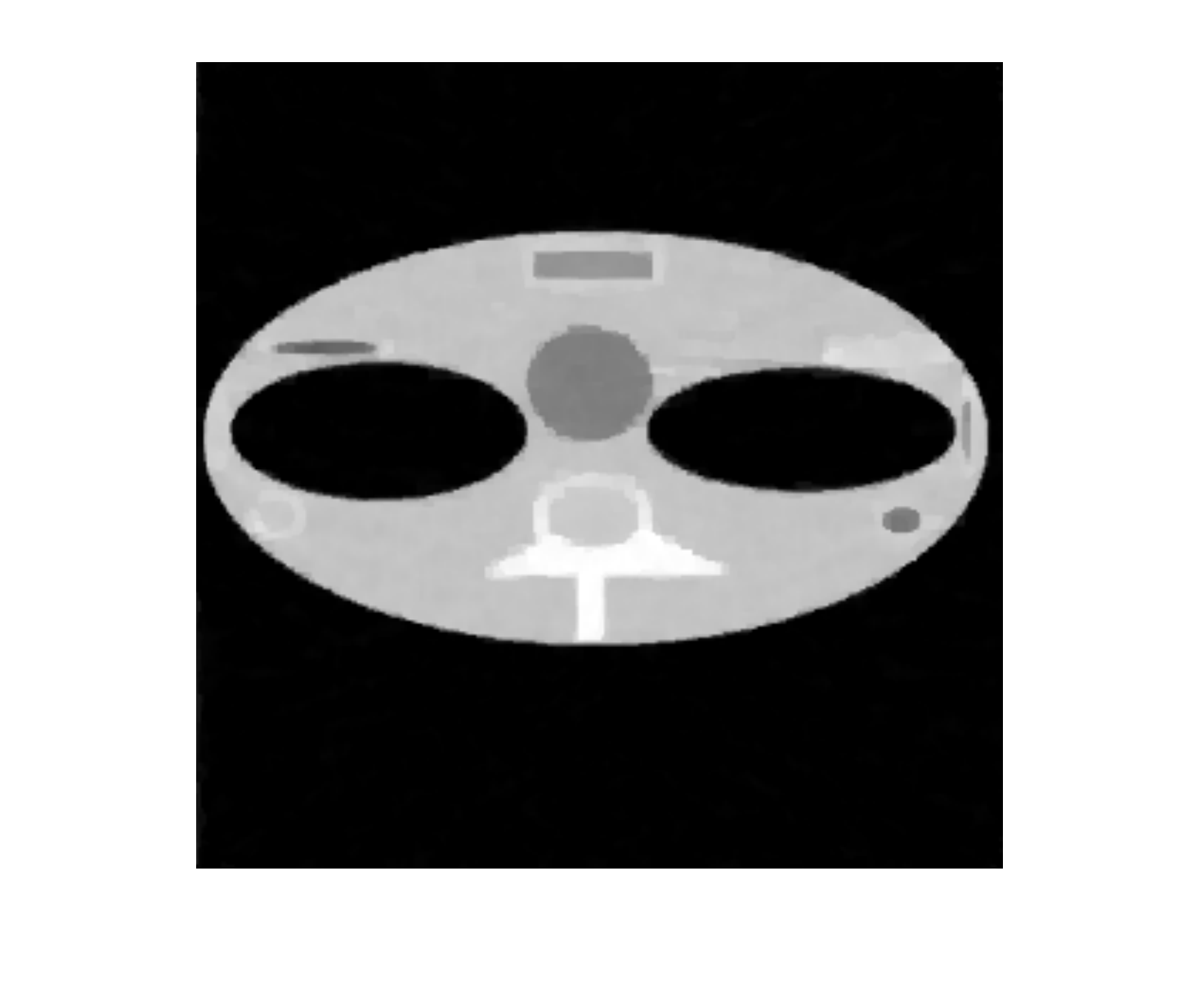}
         \caption{CT reconstruction of (a) by full update}\label{fig:pds_b}
    \end{subfigure} %
    \begin{subfigure}[t]{0.5\linewidth}
        \centering
        \includegraphics[width=1\linewidth]{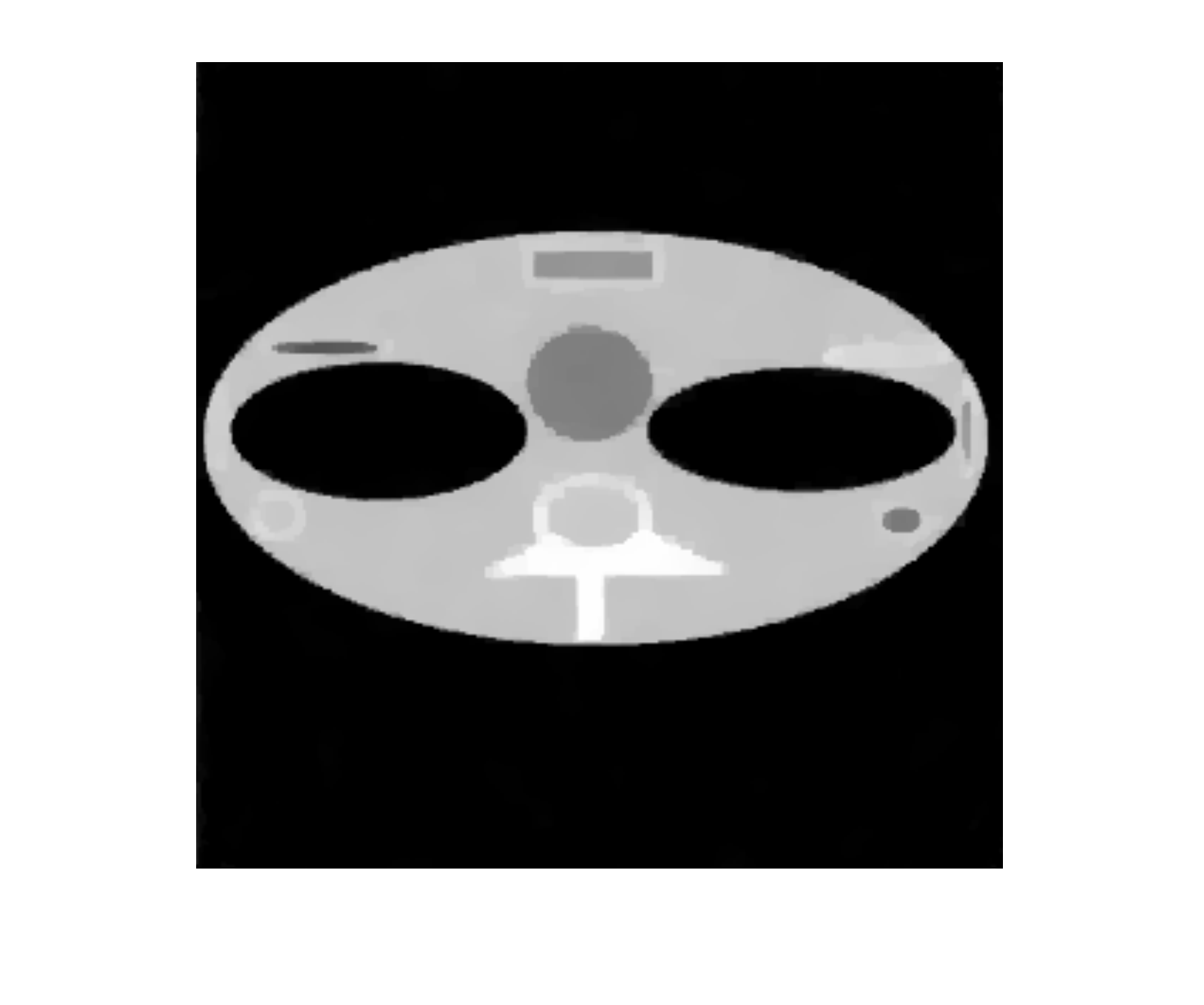}
         \caption{CT reconstruction of (a) by cyclic coordinate update}\label{fig:pds_c}
    \end{subfigure} %
    \quad
    \begin{subfigure}[t]{0.5\linewidth}
        \centering
        \includegraphics[width=1\linewidth]{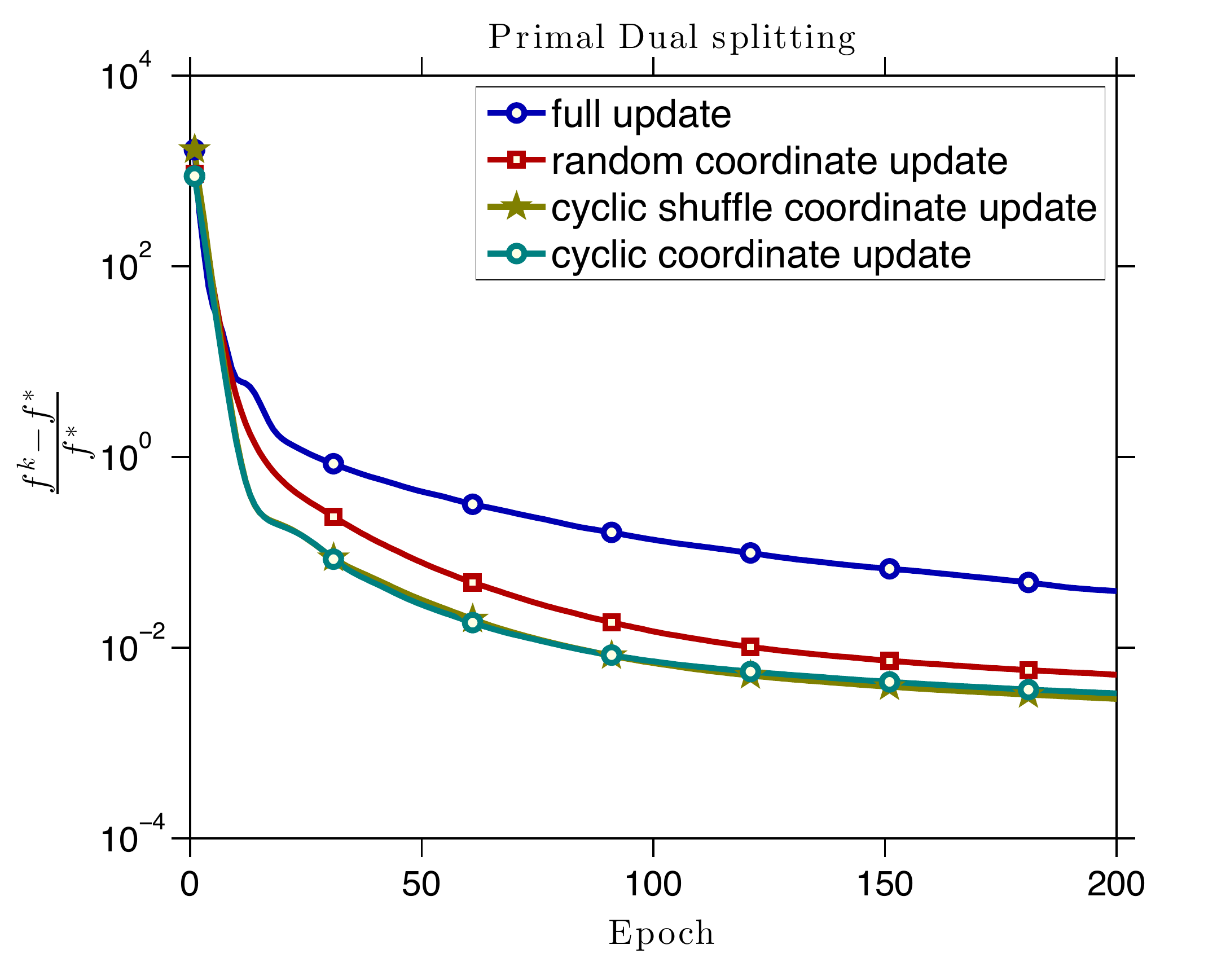}
         \caption{Objective function value}\label{fig:pds_d}
    \end{subfigure} %
    \caption{CT image reconstruction.}
    \label{fig:pds_results}
\end{figure}

Figure~\ref{fig:pds_results} depicts the results. After 200 epochs, the cyclic algorithm recovers the original image~(Figure \ref{fig:pds_a}) much better than the full update algorithm. Compare Figures~\ref{fig:pds_c} and~\ref{fig:pds_b}. 
We use the TVAL3 package\footnote{Accessed from \url{http://www.caam.rice.edu/~optimization/L1/TVAL3/} on Oct. 26, 2016.} \cite{LiYinJiangZhang2013_efficient} to obtain a high-accuracy objective value $f^*$ of~\eqref{eq:CT} and plot $\frac{f^k-f^*}{f^*}$ versus epoch in Figure~\ref{fig:pds_d}. We can see that the cyclic or cyclic shuffle versions of Algorithm~\ref{alg:cyccu} have similar performance, and they are faster than  random selection. All the three coordinate algorithms are faster than the full update algorithm.
\subsection{Nonnegative matrix factorization}

{
In this subsection, we apply Algorithm~\ref{alg:cyccu} to the nonnegative matrix factorization (NMF) problem.
NMF has attracted a great deal of attention in the last decade because it can be used to extract principal components, features, structures, 
or similarities from a large set of data or image.
Many algorithms have been developed to solve NMF (cf. \cite{history}), e.g.
multiplicative updates \cite{conv,mu1,conv2, o19},
alternating least squares \cite{fast1,fast2,fast3},
alternating non-negative least squares \cite{fast,pg,newton}, 
and block coordinate descent methods \cite{xu2013block,xu2014globally}.
}

Now let us consider the following NMF problem:
\begin{equation}
\Min_{X\in\RR_+^{n\times r},Y\in\RR_+^{m\times r}}~\frac{1}{2}\|XY^\top-M\|_F^2,
\end{equation}
where $0<r\ll \min(m,n)$ and $M\in\RR_+^{n\times m}$ are given. This problem is nonconvex.
Since the objective function is biconvex (convex in $X$ while $Y$ is fixed, and vice versa), we can apply the alternating projected gradient iteration, which we treat as a fixed-point iteration  $(X^{k+1},Y^{k+1})=T(X^k,Y^k)$:
\begin{equation}\label{eq:APG}
\begin{cases}
X^{k+1}=\max(0,X^k-\alpha_k\nabla_X f(X^k,Y^k)),\\
Y^{k+1}=\max(0,Y^k-\beta_k\nabla_Y f(X^{k+1},Y^k),
\end{cases}
\end{equation}
where $\alpha_k,\beta_k>0$ are step sizes and
\begin{equation}
\begin{cases}
\nabla_X f(X,Y)=(XY^\top-M)Y,\\
\nabla_Y f(X,Y)=(YX^\top-M^\top)X.\\
\end{cases}
\end{equation}
We partition $X,Y$ into columns $X=[X_1\quad\cdots\quad X_r]$ and $Y=[Y_1\quad\cdots\quad Y_r]$ and apply the following coordinate update to~\eqref{eq:APG} (notation: $X_{<i}:=[X_1\quad\cdots\quad X_{i-1}]$):
\begin{equation}\label{RRI}
\begin{cases}
X_i^{k+1}=\argmin_{X_i\geq 0}\frac{1}{2}\|X_i(Y_i^k)^\top+X_{<i}^{k+1}(Y_{<i}^{k+1})^\top+X_{>i}^k(Y_{>i}^k)^\top-M\|^2_F,\\
Y_i^{k+1}=\argmin_{Y_i\geq 0}\frac{1}{2}\|X^{k+1}_iY_i^\top+X_{<i}^{k+1}(Y_{<i}^{k+1})^\top+X_{>i}^k(Y_{>i}^k)^\top-M\|^2_F,
\end{cases}
\end{equation}
which appears in the recent work~\cite{ho2011descent,xu2014globally}. 

Each problem in~\eqref{RRI} has closed form solutions as follows:
\begin{align}\label{eq:RRI2}
\begin{cases}
X_i^{k+1}=\Proj_{\RR_+^n}\left[X_i^k-\frac{1}{L_{X_i}^k}\nabla_{X_i} f(X_{<i}^{k+1},X_{\geq i}^k,Y_{<i}^{k+1},Y_{\geq i}^k)\right],\\
Y_i^{k+1}=\Proj_{\RR_+^m}\left[Y_i^k-\frac{1}{L_{Y_i}^k}\nabla_{Y_i} f(X_{\leq i}^{k+1},X_{>i}^k,Y_{<i}^{k+1},Y_{\geq i}^k)\right],
\end{cases}
\end{align}
where $\nabla_{X_i} f(X,Y)=(XY^\top-M)Y_i,\nabla_{Y_i} f(X,Y)=(YX^\top-M^\top)X_i$, and $L_{X_i}^k={\|Y_i^k\|^2_2}, L_{Y_i}^k={\|X_i^{k+1}\|^2_2}$ are their corresponding Lipschitz constants. 

An issue of ~\eqref{eq:RRI2} is that $X_i^{k+1},Y_i^{k+1}$ can potentially equal  zero. Hence, we make two modifications. Firstly, we force each column of $X$ to have unit length (notice $XY^\top=(XU)(YU^{-1})^\top$ for any matrix $U\in\RR_+^{r\times r}$); secondly, we redefine $L^k_{X_i}$ as $L^k_{X_i}=\max(L_{\min},{\|Y_i^k\|^2_2})$. Consequently~\eqref{eq:RRI2} is modified to the following:
\begin{equation}
\begin{cases}
X_i^{k+1}=\Proj_{\RR_+^n\cap S^{n-1}}\left[X_i^k-\frac{1}{\min(L_{\min},{\|Y_i^k\|^2_2})}\nabla_{X_i} f(X_{<i}^{k+1},X_{\geq i}^k,Y_{<i}^{k+1},Y_{\geq i}^k)\right],\\
Y_i^{k+1}=\Proj_{\RR_+^m}\left[Y_i^k-\nabla_{Y_i} f(X_{\leq i}^{k+1},X_{>i}^k,Y_{<i}^{k+1},Y_{\geq i}^k)\right],
\end{cases}\label{ModRRI}
\end{equation}
which can still be written in the closed form; see~\cite[Appendix B]{xu2014globally}.

We implement the above coordinate update with the random, cyclic and cyclic shuffle coordinate selection rules and compare their performance with~\eqref{eq:APG}.

In our experiments, we set $n=400, m=400, q=20$ and generate $M=LR+N_o$, where the elements of $L$ and $R$ are sampled from the standard normal distribution then thresholded positively. The random noise $N_o\in\RR^{n\times m}$ is generated in the same way and scaled such that $\|N_o\|_F=10^{-3}\|LR\|_F$. The constant $L_{\min}$ is set to $0.001$. The step sizes in~\eqref{eq:APG} are set as $\alpha_k=\frac{1}{\|(Y^k)^\top Y^k\|_2},\beta_k=\frac{1}{\|(X^{k+1})^\top X^{k+1}\|_2}$.

\begin{figure}[!htb]
\centering
\begin{subfigure}[t]{0.48\linewidth}\centering
\includegraphics[height=5.5cm,width=1\linewidth]{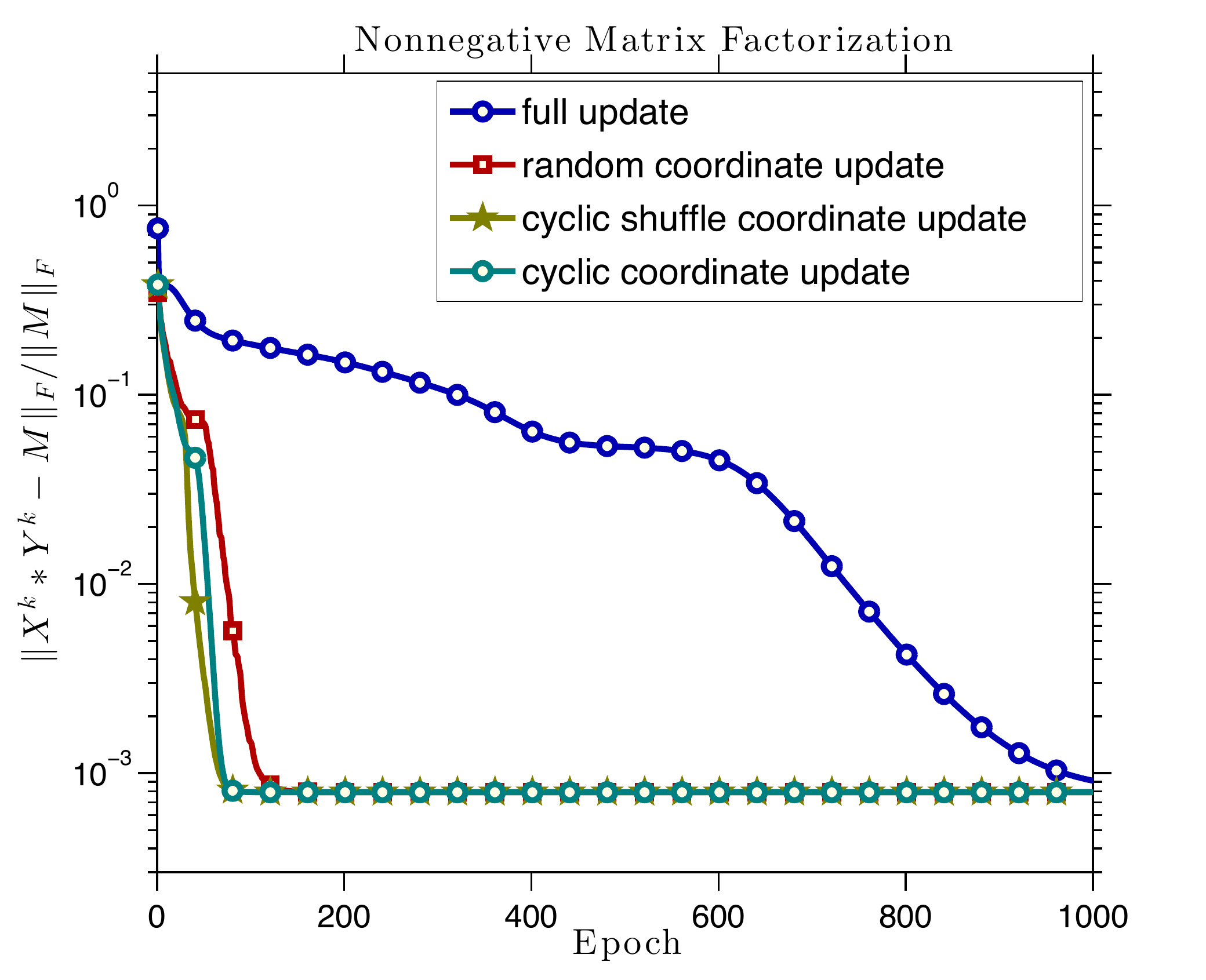}
\caption{Relative residue}
\label{fig:NMF}
\end{subfigure}
\begin{subfigure}[t]{0.48\linewidth}\centering
\includegraphics[height=5.5cm, width=1\linewidth]{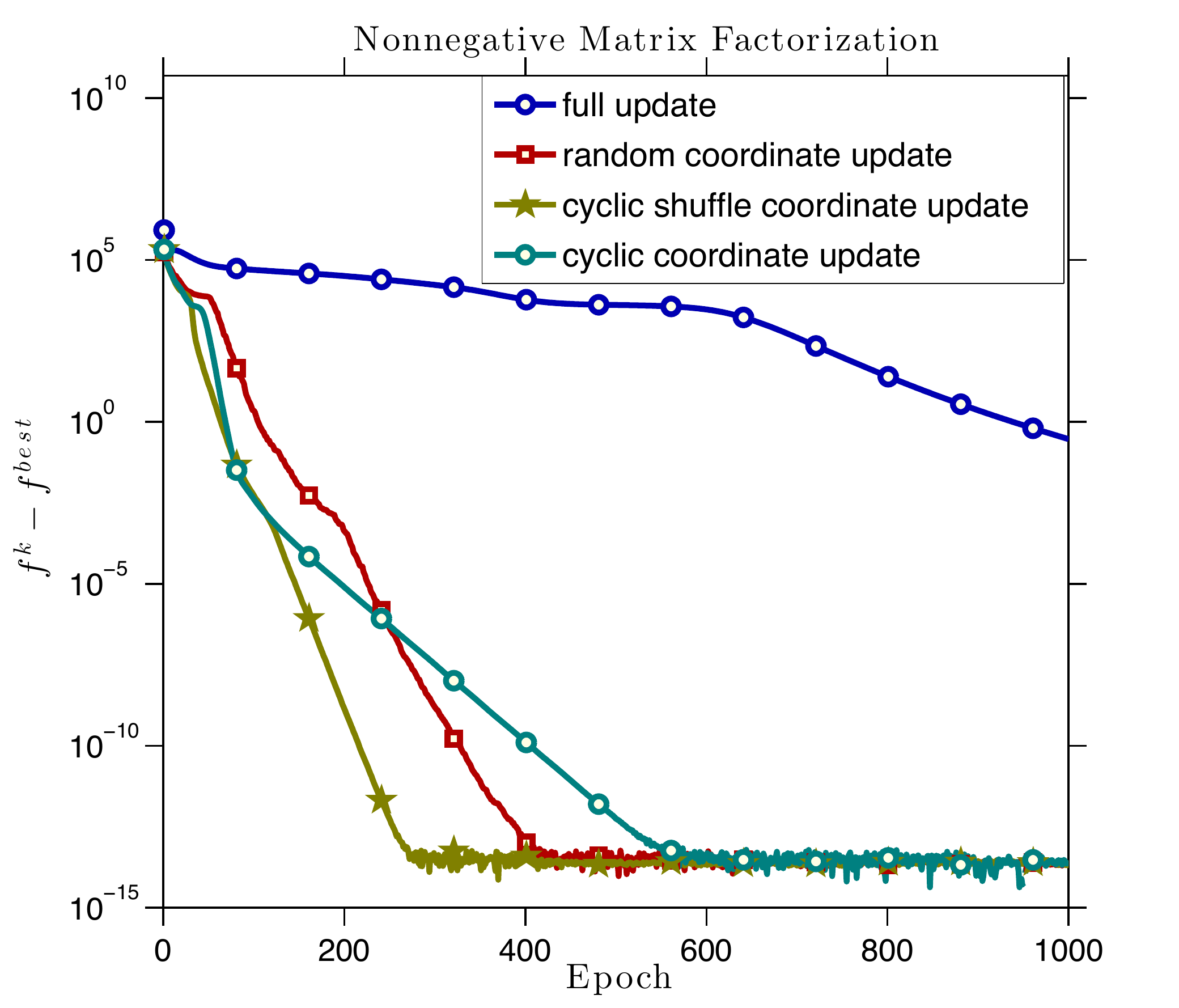}
\caption{Objective function value}
\label{fig:NMF2}
\end{subfigure}
\caption{Nonnegative matrix factorization}
\end{figure}

Figure~\ref{fig:NMF} plots the relative residue $\frac{\|X^kY^k-M\|_F}{\|M\|_F}$ versus epoch. To better compare the algorithms, we record the smallest function value $f^\text{best}$ achieved by running all algorithms for 4000 epochs and plot $f^k-f^\text{best}$ versus epoch in Figure~\ref{fig:NMF2}. The comparison results are similar, except that cyclic update becomes slightly slower than the other two after reaching a medium accuracy. 
\section{Conclusion}\label{sec:con}
We have proposed {a unified cyclic coordinate-update algorithm} for the nonexpansive fixed-point problem, established its convergence with slowly decreasing step sizes and, under a stronger condition, {its convergence} with a fixed step size. Numerical results illustrate the higher efficiency of the proposed algorithms over the traditional fixed-point iteration. Also, cyclic coordinate selection is overall more efficient than  random coordinate selection.

\section*{Acknowledgements}
We thank Robert Hannah, Zhimin Peng, Yangyang Xu, and Ming Yan for motivating discussions, as well as Qing Ling for his comments. We also thank Zhimin Peng for sharing his CT code.

\bibliographystyle{siamplain}
\bibliography{cyclic}

\begin{thebibliography}{10}

\bibitem{BauschkeCombettes2011}
{\sc H.~H. Bauschke and P.~L. Combettes}, {\em Convex {{Analysis}} and
  {{Monotone Operator Theory}} in {{Hilbert Spaces}}}, CMS Books in
  Mathematics, {Springer New York}, New York, NY, 2011.

\bibitem{beck2013convergence}
{\sc A.~Beck and L.~Tetruashvili}, {\em On the convergence of block coordinate
  descent type methods}, SIAM Journal on Optimization, 23 (2013),
  pp.~2037--2060.

\bibitem{bertsekas1983distributed}
{\sc D.~P. Bertsekas}, {\em Distributed asynchronous computation of fixed
  points}, Mathematical Programming, 27 (1983), pp.~107--120.

\bibitem{bertsekas1989parallel}
{\sc D.~P. Bertsekas and J.~N. Tsitsiklis}, {\em Parallel and distributed
  computation: numerical methods}, Prentice hall Englewood Cliffs, NJ, 1989.

\bibitem{chambolle2011first}
{\sc A.~Chambolle and T.~Pock}, {\em A first-order primal-dual algorithm for
  convex problems with applications to imaging}, Journal of Mathematical
  Imaging and Vision, 40 (2011), pp.~120--145.

\bibitem{chen2015extended}
{\sc C.~Chen, M.~Li, X.~Liu, and Y.~Ye}, {\em Extended admm and bcd for
  nonseparable convex minimization models with quadratic coupling terms:
  Convergence analysis and insights}, arXiv preprint arXiv:1508.00193,  (2015).

\bibitem{conv}
{\sc E.~C. Chi and T.~G. Kolda}, {\em On tensors, sparsity, and nonnegative
  factorizations}, SIAM Journal on Matrix Analysis and Applications, 33 (2012),
  pp.~1272--1299.

\bibitem{combettes2014forward}
{\sc P.~L. Combettes, L.~Condat, J.-C. Pesquet, and B.~C. Vu}, {\em A
  forward-backward view of some primal-dual optimization methods in image
  recovery}, in Proceedings of the 2014 IEEE International Conference on Image
  Processing (ICIP), 2014, pp.~4141--4145.

\bibitem{combettes2015stochastic}
{\sc P.~L. Combettes and J.-C. Pesquet}, {\em Stochastic quasi-fej{\'e}r
  block-coordinate fixed point iterations with random sweeping}, SIAM Journal
  on Optimization, 25 (2015), pp.~1221--1248.

\bibitem{condat2013primal}
{\sc L.~Condat}, {\em A primal--dual splitting method for convex optimization
  involving {L}ipschitzian, proximable and linear composite terms}, Journal of
  Optimization Theory and Applications, 158 (2013), pp.~460--479.

\bibitem{da2010cone}
{\sc A.~P. Da~Costa and A.~Seeger}, {\em Cone-constrained eigenvalue problems:
  theory and algorithms}, Computational Optimization and Applications, 45
  (2010), pp.~25--57.

\bibitem{mu1}
{\sc M.~E. Daube-Witherspoon and G.~Muehllehner}, {\em An iterative image space
  reconstruction algorthm suitable for volume ect}, IEEE transactions on
  medical imaging, 5 (1986), pp.~61--66.

\bibitem{DavisYin2015_threeoperator}
{\sc D.~Davis and W.~Yin}, {\em A three-operator splitting scheme and its
  optimization applications}, UCLA CAM Report 15-13,  (2015).

\bibitem{davis2014convergence}
{\sc D.~Davis and W.~Yin}, {\em Convergence rate analysis of several splitting
  schemes}, in Splitting {{Methods}} in {{Communication}}, {{Imaging}},
  {{Science}} and {{Engineering}}, R.~Glowinski, S.~Osher, and W.~Yin, eds.,
  Chapter 4, {Springer}, 2016.

\bibitem{fercoq2015coordinate}
{\sc O.~Fercoq and P.~Bianchi}, {\em A coordinate descent primal-dual algorithm
  with large step size and possibly non separable functions}, arXiv preprint
  arXiv:1508.04625,  (2015).

\bibitem{friedman2007pathwise}
{\sc J.~Friedman, T.~Hastie, H.~H{\"o}fling, R.~Tibshirani, et~al.}, {\em
  Pathwise coordinate optimization}, The Annals of Applied Statistics, 1
  (2007), pp.~302--332.

\bibitem{friedman2010regularization}
{\sc J.~Friedman, T.~Hastie, and R.~Tibshirani}, {\em Regularization paths for
  generalized linear models via coordinate descent}, Journal of statistical
  software, 33 (2010), p.~1.

\bibitem{GabayMercier1976_dual}
{\sc D.~Gabay and B.~Mercier}, {\em A dual algorithm for the solution of
  nonlinear variational problems via finite element approximation}, Computers
  \& Mathematics with Applications, 2 (1976), pp.~17--40.

\bibitem{history}
{\sc N.~Gillis}, {\em The why and how of nonnegative matrix factorization}, in
  Regularization, Optimization, Kernels, and Support Vector Machines, Chapman
  and Hall/CRC, 2014, pp.~257--291.

\bibitem{GlowinskiMarroco1975_LapproximationPar}
{\sc R.~Glowinski and A.~Marroco}, {\em Sur l'approximation, par
  \'{e}l\'{e}ments finis d'ordre un, et la r\'{e}solution, par
  p\'{e}nalisation-dualit\'{e} d'une classe de probl\`{e}mes de {Dirichlet} non
  lin\'{e}aires}, ESAIM: Mathematical Modelling and Numerical Analysis, 9
  (1975), pp.~41--76.

\bibitem{fast}
{\sc N.~Guan, D.~Tao, Z.~Luo, and B.~Yuan}, {\em Nenmf: an optimal gradient
  method for nonnegative matrix factorization}, IEEE Transactions on Signal
  Processing, 60 (2012), pp.~2882--2898.

\bibitem{conv2}
{\sc J.~Han, L.~Han, M.~Neumann, and U.~Prasad}, {\em On the rate of
  convergence of the image space reconstruction algorithm}, Operators and
  matrices, 3 (2009), pp.~41--58.

\bibitem{hanke1995convergence}
{\sc M.~Hanke, A.~Neubauer, and O.~Scherzer}, {\em A convergence analysis of
  the landweber iteration for nonlinear ill-posed problems}, Numerische
  Mathematik, 72 (1995), pp.~21--37.

\bibitem{ho2011descent}
{\sc N.-D. Ho, P.~Van~Dooren, and V.~D. Blondel}, {\em Descent methods for
  nonnegative matrix factorization}, in Numerical Linear Algebra in Signals,
  Systems and Control, Springer, 2011, pp.~251--293.

\bibitem{hong2013iteration}
{\sc M.~Hong, X.~Wang, M.~Razaviyayn, and Z.-Q. Luo}, {\em Iteration complexity
  analysis of block coordinate descent methods}, Mathematical Programming,
  (2013), pp.~1--30.

\bibitem{iusem2010proximal}
{\sc A.~N. Iusem and W.~Sosa}, {\em On the proximal point method for
  equilibrium problems in hilbert spaces}, Optimization, 59 (2010),
  pp.~1259--1274.

\bibitem{fast1}
{\sc H.~Kim and H.~Park}, {\em Sparse non-negative matrix factorizations via
  alternating non-negativity-constrained least squares for microarray data
  analysis}, Bioinformatics, 23 (2007), pp.~1495--1502.

\bibitem{fast2}
{\sc J.~Kim, Y.~He, and H.~Park}, {\em Algorithms for nonnegative matrix and
  tensor factorizations: A unified view based on block coordinate descent
  framework}, Journal of Global Optimization, 58 (2014), pp.~285--319.

\bibitem{fast3}
{\sc J.~Kim and H.~Park}, {\em Fast nonnegative matrix factorization: An
  active-set-like method and comparisons}, SIAM Journal on Scientific
  Computing, 33 (2011), pp.~3261--3281.

\bibitem{krasnosel1955two}
{\sc M.~Krasnosel'skii}, {\em Two remarks on the method of successive
  approximations}, Uspekhi Matematicheskikh Nauk, 10 (1955), pp.~123--127.

\bibitem{o19}
{\sc D.~D. Lee and H.~S. Seung}, {\em Algorithms for non-negative matrix
  factorization}, in Advances in neural information processing systems, 2001,
  pp.~556--562.

\bibitem{LiYinJiangZhang2013_efficient}
{\sc C.~Li, W.~Yin, H.~Jiang, and Y.~Zhang}, {\em An efficient augmented
  {{Lagrangian}} method with applications to total variation minimization},
  Computational Optimization and Applications, 56 (2013), pp.~507--530.

\bibitem{pg}
{\sc C.-J. Lin}, {\em Projected gradient methods for nonnegative matrix
  factorization}, Neural computation, 19 (2007), pp.~2756--2779.

\bibitem{LionsMercier1979_SplittingAlgorithms}
{\sc P.~L. Lions and B.~Mercier}, {\em Splitting {{Algorithms}} for the {{Sum}}
  of {{Two Nonlinear Operators}}}, SIAM Journal on Numerical Analysis, 16
  (1979), pp.~964--979.

\bibitem{Lu_Xiao_rbcd_2015}
{\sc Z.~Lu and L.~Xiao}, {\em {On the complexity analysis of randomized
  block-coordinate descent methods}}, Mathematical Programming, 152 (2015),
  pp.~615--642.

\bibitem{luo1992convergence}
{\sc Z.-Q. Luo and P.~Tseng}, {\em On the convergence of the coordinate descent
  method for convex differentiable minimization}, Journal of Optimization
  Theory and Applications, 72 (1992), pp.~7--35.

\bibitem{mann1953mean}
{\sc W.~R. Mann}, {\em Mean value methods in iteration}, Proceedings of the
  American Mathematical Society, 4 (1953), pp.~506--510.

\bibitem{nesterov2012rcd}
{\sc Y.~Nesterov}, {\em Efficiency of coordinate descent methods on huge-scale
  optimization problems}, SIAM Journal on Optimization, 22 (2012),
  pp.~341--362.

\bibitem{nong2009parameter}
{\sc R.~Nong and D.~C. Sorensen}, {\em A parameter free {ADI}-like method for
  the numerical solution of large scale lyapunov equations}, Computational and
  Applied Mathematics,  (2009).

\bibitem{peng2016coordinate}
{\sc Z.~Peng, T.~Wu, Y.~Xu, M.~Yan, and W.~Yin}, {\em Coordinate friendly
  structures, algorithms and applications}, Annals of Mathematical Sciences and
  Applications, 1 (2016), pp.~57--119.

\bibitem{Peng_2015_AROCK}
{\sc Z.~{Peng}, Y.~{Xu}, M.~{Yan}, and W.~{Yin}}, {\em {ARock: an algorithmic
  framework for asynchronous parallel coordinate updates}}, SIAM Journal on
  Scientific Computing, 38 (2016), pp.~A2851--A2879.

\bibitem{pesquet2014class}
{\sc J.-C. Pesquet and A.~Repetti}, {\em A class of randomized primal-dual
  algorithms for distributed optimization}, Journal of Nonlinear and Convex
  Analysis, 16 (2015), pp.~2453--2490.

\bibitem{pock2011diagonal}
{\sc T.~Pock and A.~Chambolle}, {\em Diagonal preconditioning for first order
  primal-dual algorithms in convex optimization}, in 2011 International
  Conference on Computer Vision, IEEE, 2011, pp.~1762--1769.

\bibitem{razaviyayn2013unified}
{\sc M.~Razaviyayn, M.~Hong, and Z.-Q. Luo}, {\em A unified convergence
  analysis of block successive minimization methods for nonsmooth
  optimization}, SIAM Journal on Optimization, 23 (2013), pp.~1126--1153.

\bibitem{richtarik2014iteration}
{\sc P.~Richt{\'a}rik and M.~Tak{\'a}{\v{c}}}, {\em Iteration complexity of
  randomized block-coordinate descent methods for minimizing a composite
  function}, Mathematical Programming, 144 (2014), pp.~1--38.

\bibitem{shi2016primer}
{\sc H.-J.~M. Shi, S.~Tu, Y.~Xu, and W.~Yin}, {\em A primer on coordinate
  descent algorithms}, arXiv preprint arXiv:1610.00040,  (2016).

\bibitem{siddon1985fast}
{\sc R.~L. Siddon}, {\em Fast calculation of the exact radiological path for a
  three-dimensional ct array}, Medical physics, 12 (1985), pp.~252--255.

\bibitem{tseng2001convergence}
{\sc P.~Tseng}, {\em Convergence of a block coordinate descent method for
  nondifferentiable minimization}, Journal of optimization theory and
  applications, 109 (2001), pp.~475--494.

\bibitem{vu2013splitting}
{\sc B.~C. V{\~u}}, {\em A splitting algorithm for dual monotone inclusions
  involving cocoercive operators}, Advances in Computational Mathematics, 38
  (2013), pp.~667--681.

\bibitem{warga1963minimizing}
{\sc J.~Warga}, {\em Minimizing certain convex functions}, Journal of the
  Society for Industrial and Applied Mathematics, 11 (1963), pp.~588--593.

\bibitem{wright2015coordinate}
{\sc S.~J. Wright}, {\em Coordinate descent algorithms}, Mathematical
  Programming, 151 (2015), pp.~3--34.

\bibitem{xia2011projection}
{\sc F.-Q. Xia and N.-J. Huang}, {\em A projection-proximal point algorithm for
  solving generalized variational inequalities}, Journal of optimization theory
  and applications, 150 (2011), pp.~98--117.

\bibitem{xu2013block}
{\sc Y.~Xu and W.~Yin}, {\em A block coordinate descent method for regularized
  multiconvex optimization with applications to nonnegative tensor
  factorization and completion}, SIAM Journal on imaging sciences, 6 (2013),
  pp.~1758--1789.

\bibitem{xu2014globally}
{\sc Y.~Xu and W.~Yin}, {\em A globally convergent algorithm for nonconvex
  optimization based on block coordinate update}, arXiv preprint
  arXiv:1410.1386. To appear in Journal of Scientific Computing,  (2014).

\bibitem{yuan2010comparison}
{\sc G.-X. Yuan, K.-W. Chang, C.-J. Hsieh, and C.-J. Lin}, {\em A comparison of
  optimization methods and software for large-scale l1-regularized linear
  classification}, Journal of Machine Learning Research, 11 (2010),
  pp.~3183--3234.

\bibitem{newton}
{\sc R.~Zdunek and A.~Cichocki}, {\em Non-negative matrix factorization with
  quasi-newton optimization}, in International Conference on Artificial
  Intelligence and Soft Computing, Springer, 2006, pp.~870--879.

\end{thebibliography}
\end{document}